\documentclass{elsarticle}

\usepackage[utf8]{inputenc}
\usepackage[T1]{fontenc}
\usepackage{lmodern}
\usepackage{amsmath}
\usepackage{amssymb}
\usepackage{amsthm}
\usepackage{mathrsfs}
\usepackage{graphicx}
\usepackage{pst-all}
\usepackage{tikz}
\usepackage{marvosym}
\usetikzlibrary{shapes}
\usepackage{epic}
\usepackage{latexsym} 
\usepackage{amsmath} 
\usepackage{epsfig}
\usepackage{amssymb} 
\usepackage{enumerate}
\usepackage{epic}
\usepackage[mathscr]{euscript}
\usepackage{hyperref}

\newcommand{\I}{\mathcal{I}}

\begin{document}
\begin{frontmatter}

\title{Bridging the gap between rooted and unrooted phylogenetic networks}
\author{P. Gambette}
\address{Universit\'e Paris-Est, LIGM (UMR 8049), UPEM, CNRS, 
ESIEE, ENPC, F-77454,
Marne-la-Vall\'ee, France.}
\author{K.T. Huber, G.E. Scholz}
\address{School of Computing Sciences,
University of East Anglia, UK.} 

\date{\today}

\newtheorem{theorem}{Theorem}[section]
\newtheorem{lemma}[theorem]{Lemma} \newtheorem{sublemma}[theorem]{}
\newtheorem{corollary}[theorem]{Corollary}
\newtheorem{proposition}[theorem]{Proposition}
\newtheorem{example}[theorem]{Example}
\newtheorem{definition}[theorem]{Definition}

\newcommand{\cQ}{{\mathcal Q}}
\newcommand{\cC}{{\mathcal C}}
\newcommand{\cS}{{\mathcal S}}
\newcommand{\cL}{{\mathcal L}}
\newcommand{\cT}{{\mathcal T}}
\newcommand{\cB}{{\mathcal B}}

\newcommand{\pf}{\noindent{\em Proof: }}
\newcommand{\epf}{\hfill\hbox{\rule{3pt}{6pt}}\\}

\begin{abstract}
The need for structures capable of accommodating complex evolutionary
signals such as those found in, for example, wheat has fueled research into
phylogenetic networks. Such structures generalize the standard 
phylogenetic tree model by also allowing cycles and have been
introduced in rooted and unrooted form. In contrast to phylogenetic 
trees, however, surprisingly little is known about the interplay between both
types thus hampering our ability to make much needed progress for rooted
phylogenetic networks by drawing on insights from their much better
understood unrooted counterparts. 
Unrooted phylogenetic networks 
are underpinned by split systems and by focusing on them
we establish a first link between both types. More precisely, we
develop a link between 1-nested phylogenetic networks which are
examples of rooted phylogenetic networks and the well-studied median
networks (aka Buneman graph)
which are examples of unrooted phylogenetic networks. In particular, we
show that not only can a 1-nested network be obtained from a median
network but also that that network is, in a well-defined sense, optimal. 
Along the way, we characterize circular split systems in terms of the novel
$\mathcal I$-intersection closure of a split system and establish the
1-nested analogue of the fundamental ``Splits Equivalence Theorem'' for
phylogenetic trees. 
\end{abstract}



\end{frontmatter}


\section{Introduction}\label{sec:intro}

A widely accepted evolutionary scenario for 
some economically important crop plants such as wheat
is that their evolution has been shaped by complex 
reticulate processes \cite{M-et-al15}. 
The need for structures capable of representing
the telltale signs left behind by them has 
fueled research into phylogenetic networks which generalize the
commonly used phylogenetic tree model for studying molecular evolution 
(see Fig~\ref{fig:intro-network} for examples
and Section~\ref{sec:prelim} for formal definitions). 
However, despite many years of 
research into phylogenetic networks (see e.\,g.\,the graduate
text books \cite{G14,HRS10})
important questions concerning their structure 
have remained unanswered so far (see e.\,g.\,\cite{GH12}). 
\begin{figure}[h]
\begin{center}
\includegraphics[scale=0.5]{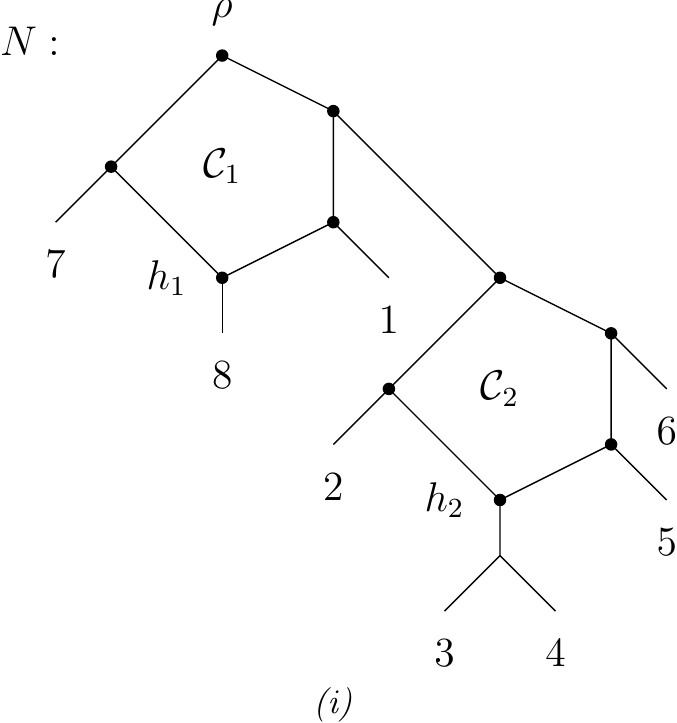}
\,\,\,\,\,\,\,\,\,
\includegraphics[scale=0.5]{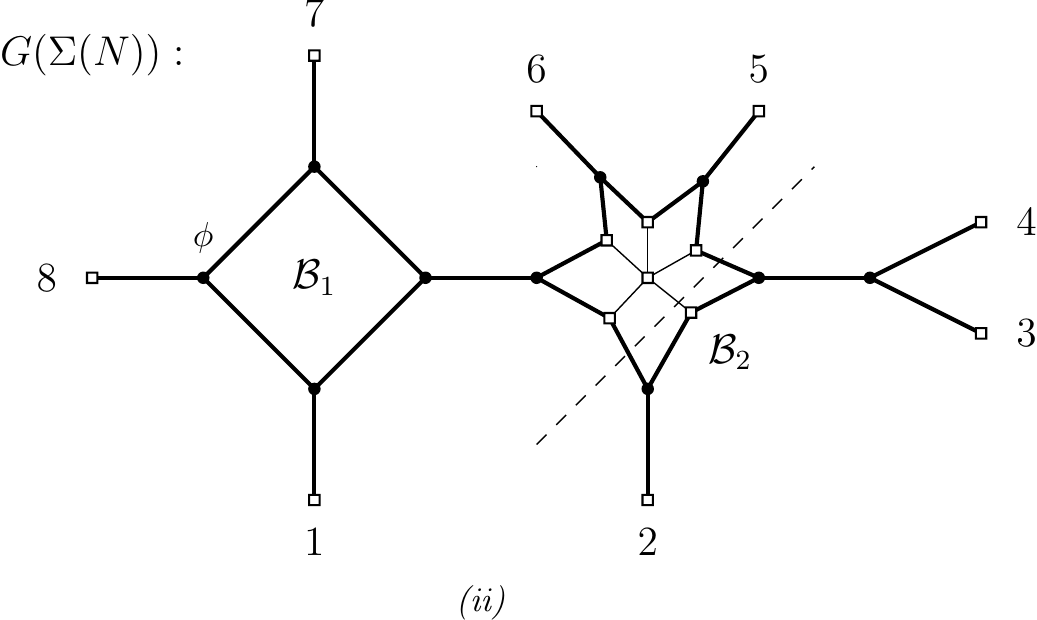}
\caption{\label{fig:intro-network} 
(i) A rooted phylogenetic
network in the form of a level-1 network $N$ on $X=\{1,\ldots,8\}$.
The root is labeled by $\rho$, all edges are directed downwards,
and vertices $h_1$ and $h_2$ represent hypothesised reticulate
evolutioanry events..
 (ii) An unrooted phylogenetic network on $X$ 
in the form of a median network on the
split system induced by the underlying graph of $N$. The dashed line
indicates the split $234|15678$ -- see text for
details.
}
\end{center}
\end{figure}

Mimicking the situation for phylogenetic trees which can either be
rooted or unrooted, phylogenetic networks have been introduced 
in terms of rooted directed acyclic graphs
whose outdegree zero vertices correspond to the taxa (e.g. species)
of interest and also in terms of representations of 
{\em splits systems}, that is, collections of 
bipartitions of the taxa set in question.
In general, the former seem more attractive as they allow the inclusion of 
directionality and thus readily lend themselves to an
interpretation within an evolutionary context. However, 
they suffer from the fact 
that they are generally poorly understood from a combinatorial point of view 
thus hampering our ability to design powerful reconstruction
algorithms for them. Unrooted counterparts of rooted phylogenetic
 networks include the 
popular NeighborNet approach \cite{BM04} as well as median networks 
\cite{BFSR95,BYBSK09}
(sometimes also called {\em Buneman graphs} -- see e.\,g.\,\cite{DHKM11}) 
and related structures (see e.\,g.\,\cite{BD07,BFR99}) and 
the methodology underpinning them is far more advanced. This is
in part due to a rich body of literature surrounding
such graphs which have appeared under various
guises such as co-pair hypergraphs \cite{B89,BG91} 
and have been studied in terms of 
median algebras \cite{BH83}, 1-skeletons of CAT(0) cubical complexes 
\cite{BC96}, retracts of hypercubes \cite{B84}, 
sets of solutions of 2-SAT formulas \cite{BH83}, tight spans of
metric spaces (see e.g. \cite{DHM02} and also the more recent
text book \cite{DHKMS12} 
and the references therein), and
S2 binary convexities \cite{v-D-V84} (see also \cite{KM99}
for a review of median graphs). In an independent line of research, numerous
deep results have also been obtained for the special case that
the (unrooted) phylogenetic network is in fact a tree (see 
e.\,g.\,\cite{SS03,DHKMS12}). 

From a combinatorial point of view, rooted and unrooted phylogenetic trees
can be thought of as certain {\em cluster systems} 
(i.\,e.\,collections of non-empty subsets
of the tree's leaf set) and split systems, respectively, 
and the  Farris Transform
allows one to readily translate between both by ignoring/identifying 
a root vertex and edge-directions (see e.\,g.\,\cite{DHKMS12}
for details). As it turns out, the more
general rooted and unrooted phylogenetic networks
also induce cluster systems and split systems, respectively.
Thus, it is conceivable that a similar strategy 
could be used to foster our understanding of
rooted phylogenetic networks. For this to work however, 
some care needs to be taken since the
graph obtained from a rooted phylogenetic network by ignoring its
root and edge-directions can contain odd
length cycles and thus is not a (splits based) phylogenetic network as
the inner workings of such network require them to  
only contain cycles of even length. 

Intriguingly, any rooted phylogenetic network $N$ also induces a split system 
$\Sigma(N)$ by taking minimal edge cuts in the {\em underlying
graph} $U(N)$ of $N$ i.\,e.\,the graph obtained from $N$ by 
ignoring edge directions and  suppressing
the root in case its out-degree in $N$ is two (see \cite{BC10} and also
\cite{GBP12}). Such graphs can be thought of as intermediate steps
between rooted and unrooted phylogenetic networks and
although they do not contain directionality information 
they still provide valuable information on the number of reticulate
evolutionary events which is a difficult problem of interest in its own right.
Furthermore, $\Sigma(N)$
can also be represented in terms of an unrooted phylogenetic  
network (although 
the way $\Sigma(N)$ is displayed by such a network is fundamentally
different from the way $\Sigma(N)$ is displayed by $N$ -- see 
Figs~\ref{fig:intro-network} and 
\ref{fig:example} for an illustration of this fact in terms of the
Buneman graph $G(\Sigma(N))$ associated to $\Sigma(N) $ where the
way the split $234|15678$ is
displayed by $N$ and  $G(\Sigma(N))$ is indicated in terms of a dashed line). 

\begin{figure}[h]
\begin{center}
\includegraphics[scale=0.5]{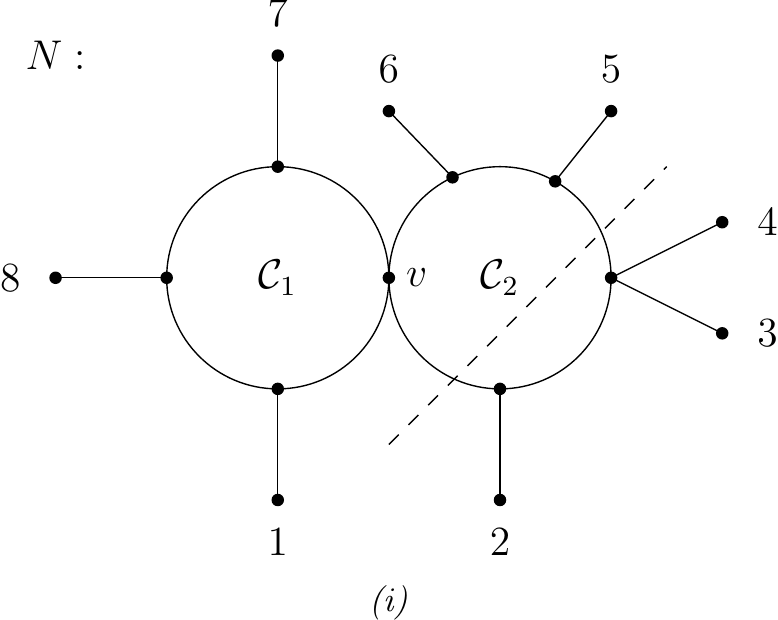}\,\,\,\,\,\
\includegraphics[scale=0.5]{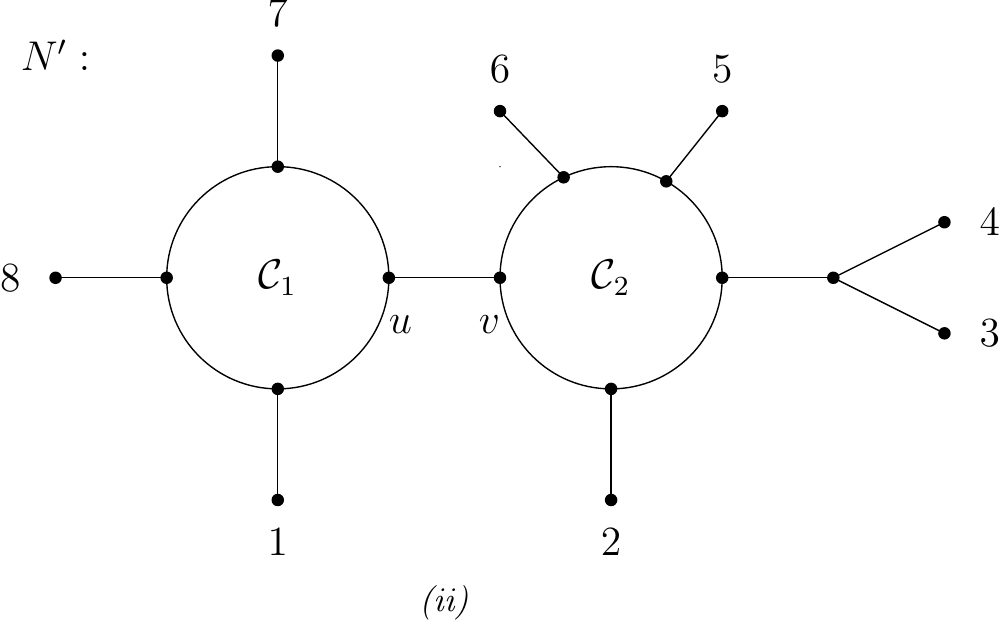}
\caption{\label{fig:example} 
(i) A 1-nested network for which the induced split system is $\Sigma(N')$.
As in Fig.~\ref{fig:intro-network}, the dashed line
indicates the split $234|15678$
(ii) The underlying graph of the rooted level-1
network $N$ depicted in Fig.\ref{fig:intro-network}.
}
\end{center}
\end{figure}

In general, the
split system induced by a rooted phylogenetic network $N$ on some set $X$
can be very complicated. However, in case $N$  is level-1
\cite{JS04,JS08} which essentially means that no two cycles in
 $U(N)$ share a vertex
then the induced split system $\Sigma(N)$ is {\em circular} 
(i.\,e.\,the elements of $X$ can be arranged around a 
cycle $C$ so that the 
split system induced on $X$ by deleting any two edges of $C$ contains
in $\Sigma(N)$) \cite{GBP12}.
This property is central to 
the aforementioned popular NeighborNet approach
and particularly attractive as it guarantees any such split system to be 
representable in the plane in terms of an unrooted phylogenetic
networks without crossing edges. 

Level-1 networks and also the more general  rooted 1-nested networks
(i.\,e.\,the networks obtained by replacing the 
requirement of vertex disjointness between cycles in the underlying graph
in the definition of a level-1 network
by allowing cycles to share at most one vertex -- see \cite{RV09})
 have attracted a considerable amount of attention
in the literature (see e.\,g.\,\cite{HM13} and the references therein).
For ease of readability and also 
reflecting the fact that the focus of this paper lies on
understanding split systems induced by rooted 
level-1 networks we will from now on, unless indicated otherwise,
 refer to the underlying graph of a 
rooted 1-nested network as a 1-nested network.

As is easy to see, any circular
split system on some set $X$ can be represented in terms of a level-1
network $N$ on $X$ by taking the unique cycle of $N$ to
be the aforementioned cycle $C$, attaching to each vertex $v$ of
$C$ a pendant edge $e$ and shifting the element of $X$  labelling $v$
to the degree one vertex of $e$. Although structurally very simple,
these types of networks are of interest in their own right as
they are special types of so called unicyclic networks \cite{SS06} (see
also \cite{RL04} for a biological example)
which have been related to the tree arrangement problem
 in \cite{SS06}. However, as the level-1 network 
depicted in Fig.~\ref{fig:example}(ii) indicates for the split system $\Sigma$
comprising of all splits of the form $x|X-x$\footnote{For ease of readability,
we represent a split $\{A,B\}$ of a set $X$ as $A|B$ where
the order of $A$ and $B$ is irrelevant. Also if $A=\{x_1,\ldots,x_k\}$
and $B=\{x_{k+1},\ldots, x_n\}$ for some $1\leq k\leq n-1$
we write $x_1\ldots x_k|x_{k+1}\ldots x_n$ rather than
$\{x_1,\ldots,x_k\}|\{x_{k+1},\ldots, x_n\}$. }
 where $x\in X:=\{1,\ldots,8\}$ 
and the splits $81|234567$, 
$78|123456$, $781|23456$, $234|56781$, $34|567812$, $345|67812$,
$2345|6781$, $3456|7812$ and $56|78123$, the resulting level-1 network is 
generally not optimal as it displays 
a total of ${|X| \choose 2}$ distinct splits of $X$ whereas the level-1 network
$N$ depicted in that figure also displays all splits of $\Sigma$ and
 postulates fewer additional splits.
Furthermore, the 1-nested network 
pictured  in Fig.~\ref{fig:example}(i) 
also displays $\Sigma$ and so does the subgraph in terms of
bold edges of the Buneman graphs 
$G(\Sigma)$ 
pictured Fig.~\ref{fig:intro-network}(ii).
%
%

As it turns out, this is not a coincidence since
Theorem~\ref{theo:buneman} combined with 
Corollary~\ref{cor:equivalence} ensures that, up to isomorphism
and a mild condition, the network obtained from the Buneman graph pictured
in Fig.~\ref{fig:intro-network}(ii)  by
deleting the central vertex in $\mathcal B_2$, its incident edges and 
suppressing the resulting degree two vertices is in fact optimal. 
Corollary~\ref{cor:equivalence} itself may be
the 1-nested analogue of
the fundamental ``Splits Equivalence Theorem'' for unrooted
phylogenetic trees \cite[Theorem 3.1.4]{SS03} and
is a consequence of  Theorem~\ref{theo:inclu}. The
purpose of that theorem is to establish that the 
novel $\mathcal I$-intersection
closure of a split system $\Sigma$ can be used to obtain a 1-nested 
network $N$ whose induced split system $\Sigma(N)$ does not only
contains $\Sigma$ but is also 
minimum.
Given that a closure is not known for rooted phylogenetic networks
it might be interesting to see if our
closure could be used to this effect
(see also Section~\ref{sec:open-probs}).

 As an
important stepping stone for establishing Theorem~\ref{theo:inclu},
we characterize split systems that are induced by a 
1-nested network in terms of their $\mathcal I$-intersection closure
(Theorem~\ref{theo:equal}) and also when a circular split system that
is $\mathcal I$-intersection closed is (set-inclusion) maximal 
(Theorem~\ref{theo:ordun}).
We remark in passing that Theorem 
\ref{theo:inclu} complements a result by Dinitz and Nutov \cite{DN95,DN-ms}
 who characterized
split systems that can be represented by 1-nested networks in terms of 
``crossings'' and ``cactus models'' and a result by Brandes and
Cornelsen  \cite{BC10} who gave an 
$\mathcal O(f+|X|+|\Sigma|)$ algorithm for deciding if a split system $\Sigma$
on $X$ can be represented by a 1-nested network 
or not and if so constructing such a network where 
$f\leq |X|\times|\Sigma|/2$. Since it is not
difficult to see that crossing split systems are a particular type of
split system that is $\mathcal {I}$-intersection closed their result
is viewable as a consequence.

The outline of the paper is as follows. In the next section we introduce
some relevant basic terminology such as level-1 and 1-nested network. 
In Section~\ref{sec:intersection}, we introduce and study
the $\mathcal I$-closure operation which lies at the
heart of Theorem~\ref{theo:equal}. In  Section~\ref{sec:cir-ord},
we turn our attention
to (set-inclusion) maximal circular split systems
and  establish
 Theorems~\ref{theo:ordun} and \ref{theo:inclu}. 
As part of this, we characterize such collections
in terms of a property of the incompatibility graph
of a split system which we define in that section.
Using insights into the structure of the Buneman
graph presented in \cite{DHKM11}, we establish Theorem~\ref{theo:buneman}
in Section~\ref{sec:buneman}. We conclude with some open problems in 
Section~\ref{sec:open-probs}.

\section{Preliminaries}\label{sec:prelim}
In this section, we present relevant basic definitions concerning 
splits and
phylogenetic networks. Throughout the paper, we assume that 
 $X$ is a finite set with  $n\geq 3$ elements and that, unless
stated otherwise, split systems are non-empty. 

\subsection{Splits and split systems}
For all subsets $A\subseteq X$, we put
 $\bar{A}=X-A$. Furthermore, 
for all elements $x\in X$ and all splits $S$ of $X$, we denote 
by $S(x)$ the element of $S$ that contains $x$.  
The {\em size} of a split $A|B$ is defined as $\min\{|A|,|B|\}$
and a split $S$ is called \emph{trivial} if its size is one.
Two splits $S_1$ and $S_2$ of $X$ are called {\em compatible}
if there exists some $A_1\in S_1$ and some $A_2\in S_2$ such that
$A_2\subsetneq A_1$ and {\em incompatible} otherwise.
More generally, a split system
$\Sigma$ on $X$ is called {\em compatible} if any
two splits in $\Sigma$ are compatible and {\em incompatible}
otherwise.

%

Suppose $x_1,x_2,\ldots,x_n,x_{n+1}:=x_1$
is a circular ordering of the elements of $X$
(where we take indices modulo $n$). Then
for all $i,j\in\{1,\ldots,n\}$ we call the
subsequence $x_i,x_{i+1},\ldots, x_j$ the {\em interval}
from $x_i$ to $x_j$ and denote it by $[x_i, x_j]$.  
We say that a split system $\Sigma$ on $X$ is 
\emph{circular} if there exists a circular ordering
$x_1, x_2,..., x_n, x_{n+1}=x_1$ of the elements of $X$ such 
that for every split $S=A|B\in \Sigma$ there exists an 
$i,j\in\{1,\ldots,n\}$ such that $A=[x_i,x_j]$
and $B=[x_{j+1}, x_{i-1}]$. Note that there are $(n-1)!$ circular orderings
for $X$ and that a circular split system on $X$ has size 
at most $n(n-1)/2$.

\subsection{Phylogenetic networks}

Suppose $G$ is a simple connected graph. 
Then the \emph{cyclomatic number} of $G$ 
is the minimum number of edges that need to be 
removed from $G$ to obtain a tree. A \emph{cut-edge} of $G$ 
is an edge $e$ whose removal disconnects $G$. We call $e$ 
 \emph{trivial} if it is incident to a {\em leaf} $v$ of $G$, 
that is, the degree of $v$ is one. 

We call a simple connected graph $N$ a
\emph{phylogenetic network (on $X$)} if $X$ 
is the set of leaves of $N$, every other vertex has degree 
at least three, and every cycle has length at least four.  The reason
for the latter requirement is that a cycle of length three displays the same
split system as the star tree with three leaves (but not the same
multi-set of splits) which is undesirable from a uniqueness point of view.
As for rooted phylogenetic networks, we call a 
phylogenetic network $N$ \emph{simple}
if all cut-edges of $N$ are trivial, {\em level-1} 
if every vertex but the leaves have degree three and
the maximum cyclomatic number of the bridgeless 
connected components of $N$ is one 
\cite{GBP12}, and, inspired by \cite{RV09},  {\em 1-nested} if
$N$ can be obtained from a level-1 network by collapsing 
(non-trivial) cut-edges.
For example, the graph $N'$ depicted in Fig.~\ref{fig:example}(i)
is a level-1 network on $X=\{1,\ldots, 8\}$ and the
graph pictured in Fig.~\ref{fig:example}(ii) is a 1-nested
network on $X$ as it can be obtained from $N'$ by collapsing the edge
$\{u,v\}$. 

Finally, we say that two phylogenetic networks $N$ and $N'$ on $X$ are 
{\em isomorph}
if there exists a graph isomorphism between $N$ and $N'$ that is the 
identity on $X$.

\subsection{Displaying splits}
Suppose that $N$ is a phylogenetic network on  $X$.
Then we say that a split $S=A|B$ of $X$ is \emph{displayed} by 
$G$ if there exists a set-inclusion minimal cut of $G$, that is, a set 
$E_S$ of edges of $G$ such that the deletion of the edges
in $E_S$ disconnects $G$ into two connected 
components, one of whose set of leaves is $A$ and the other is $B$.
More generally, we say that a split system $\Sigma$ is {\em displayed} by $N$
if every split of $\Sigma$ is displayed by $N$, that is,
$\Sigma\subseteq \Sigma(N)$. Also, 
we say that a split $S\in \Sigma(N)$ is {\em displayed 
by a cycle $C$} of $ N$ if $E_S$ is contained in the edge set of $C$.

Note that in case $N$ is a 1-nested network and $S\in \Sigma(N)$ 
then $|E_S|\in \{1,2\}$. 
Furthermore, note that if $e=\{u,v\} $ is the unique element
in $E_S$ and neither $u$ nor $v$ is contained in a cycle of $N$ then
$e$ must be a cut-edge of $N$ and the multiplicity of the split $S_e$
induced by deleting $e$ is one. Note also that if
$e=\{u,v\} $ is a cut-edge of $N$ where $u$ or $v$ is
contained in a cycle $C$ of $N$, say $u$, then $S_e$ is also induced
by deleting the edges of $C$ incident with $u$. Thus, the multiplicity
of a split induced by a 1-nested network can either be one, 
two, or three. 
We call
a split $S$ of multiplicity two or more in a maximal partial-resolution
of $N$ an {\em m-split} of $N$
(or more precisely of $C$ if $C$ is the cycle of $N$
that displays $S$). Finally,
note that the split system induced by a 1-nested network $N$ on $X$
is the same as the resolution of $N$ to a 1-nested network by repeatedly 
applying the following two replacement operations (and their complements
which we denote by (R1') and (R2'), respectively):
\begin{enumerate}
\item[(R1)] a vertex $v$ of a cycle of $N$ incident with $l\geq 2$ 
non-cycle edges $e_1,\ldots,e_l$ is replaced by an edge one of whose vertices is
$v$ and the other is incident with $e_1,\ldots,e_l$ and vice versa, and 
\item[(R2)] a cut-vertex $v$ shared by two cycles $C_1$ and $C_2$ is replaced
by a cut-edge one of whose vertices  is contained in $C_1$ and the 
other in $C_2$.
\end{enumerate}
However, the multi-sets of splits induced by both networks are clearly 
different. We call the vertex $v$ in (R1) or (R2) {\em partially-resolved}.
More generally, we call a
1-nested network $N'$ 
a {\em partial-resolution} of a 1-nested network $N$ if $N'$ can be
obtained from $N$ by partially resolving vertices of $N$.
Moreover, we call a partial-resolution $N'$ of $N$ a {\em maximal 
partial-resolution} of $N$ if the only way to obtain a partial-resolution
of $N'$ is to apply (R1') or (R2'). In this case, we also call $N'$
{\em maximal partially-resolved}.

To illustrate some of these definitions consider the 1-nested network $N$
on $X=\{1,\ldots, 8\}$ depicted in  Fig.~\ref{fig:example}(i). 
Then, the splits $7|X-\{7\}$, $8|X-\{8\}$, $1|X-\{1\}$ and 
$781|23456$ are displayed by $N$. In fact, they are 
m-splits for the cycle $\mathcal C_1$ 
of $N$. 
Furthermore, $N'$ is a partial-resolution of $N$  and 
$\overline{\Sigma(N)}$ only contains splits of multiplicity one or two. 

 
\section{Characterizing of 1-nested networks in terms of
$\mathcal I$-intersections}\label{sec:intersection}

In this section, we introduce and study the $\mathcal I$-intersection
 closure of a split system which turns out
to be key for our characterization of
1-nested networks in terms of split systems which we 
present in Theorem~\ref{theo:equal}. 
We start with introducing the concept of an intersection
between splits.

Suppose $S_1$ and $S_2$ are two distinct splits of $X$ and
 $A_i \in S_i$, $i=1,2$,
such that $A_1 \cap A_2 \neq \emptyset$. Then we call the
split $A_1 \cap A_2 | \bar{A}_1 \cup \bar{A}_2$ of $X$ associated
to $\{S_1,S_2\}$ an {\em intersection of $S_1$ and $S_2$ (with respect to 
$A_1$ and $A_2$)}. We denote the
set of all splits obtained by taking intersections
of  $S_1$ and $S_2$ by $int(S_1,S_2)$ and write
$int(S_1,S_2)$ rather than $int(\{S_1,S_2\})$. 
Furthermore, if $S_1$ and $S_2$ are
incompatible then we refer to the intersection of $S_1$ and $S_2$
as {\em incompatible intersection}, or \emph{$\mathcal{I}$-intersection} 
for short, and denote it by  $\iota(S_1,S_2)$ rather than $int(S_1,S_2)$.

Clearly, if
$S_1$ and $S_2$ are compatible then $|int(S_1,S_2)|=3$
and  $S_1,S_2\in int(S_1,S_2)$.
However, if $S_1$ and $S_2$ are incompatible then 
  $\iota(S_1,S_2)$
 is compatible and of size four, $S_1,S_2\not\in\iota(S_1,S_2)$,
and every split in $\iota(S_1,S_2)$ is compatible with $S_1$ and $S_2$. 
See Fig.~\ref{fig:schema1} an illustration. 

\begin{figure}[h]
\begin{center}
\includegraphics[scale=0.8]{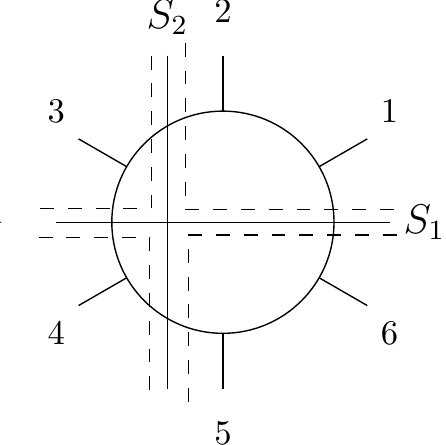}
\label{fig:schema1}
\caption{For a simple level-1 network on $\{1,\ldots, 6\}$
we depict the splits $S_1$ and $S_2$ in terms of
two straight bold lines  and 
the four splits   that make up $\iota(S_1,S_2)$ in terms of four dashed lines.}
\end{center}
\end{figure}

Fig.~\ref{fig:schema1} suggests that every split 
in $\iota(S_1,S_2)$ is displayed by the same cycle
that displays $S_1$ and $S_2$. That this is indeed the case is
the purpose of Proposition~\ref{prop:intnet}. To state it in its full 
generality we next associate to a split system $\Sigma$ of $X$ 
the {\em intersection closure} $Int(\Sigma)$ of $\Sigma$, 
that is, $Int(\Sigma)$ is a 
(set-inclusion) minimal split system that contains 
$\Sigma$ and is closed by intersection.
For example, for $\Sigma=\{12|345, 23|451\}$
we have $Int(\Sigma)=\Sigma\cup\{1|2345, 2|3451, 
3|4512, 13|452, 123|45\}$.

We start our analysis of $Int(\Sigma)$ with remarking that 
$Int(\Sigma)$ is indeed a closure, that is, $Int(\Sigma)$ trivially satisfies
the following three properties
\begin{enumerate}
\item[(C1)] $\Sigma\subseteq Int(\Sigma)$.
\item[(C2)] $Int(Int(\Sigma))= Int(\Sigma)$.
\item[(C3)] If $\Sigma'$ is a split system on $X$ for which
$\Sigma\subseteq \Sigma'$ holds then 
$Int(\Sigma)\subseteq Int(\Sigma')$.
\end{enumerate}

The proof of the next lemma is a straight forward consequence 
of our definitions.

\begin{lemma}
Suppose $S_1$, $S_2$, and $S_3$ are three splits on $X$,
such that both
$\{S_1, S_2\}$ and $\{S_2,S_3\}$ are incompatible.
Then there exists at least two splits in $\iota(S_1, S_2)$ that 
are incompatible with $S_3$.
\end{lemma}
%

The next lemma implies that the intersection closure
of a split system is well-defined.

\begin{lemma}\label{lem:closure}
Suppose $\Sigma$ is  a split system on $X$
and  $\Sigma' $ is a further (set-inclusion) minimal superset of 
$\Sigma$ that is closed by intersection.  
Then $\Sigma' =Int(\Sigma)$ must hold.
\end{lemma}
\begin{proof}
Since $\Sigma'$ contains $\Sigma$ and is intersection closed 
we can obtain $\Sigma'$ via  a (finite) sequence 
$\Sigma=\Sigma_0\subsetneq \Sigma_1\subsetneq \Sigma_2\subsetneq \ldots
\subsetneq \Sigma_k=\Sigma'$, $k\geq 1$,
of split systems $\Sigma_i$ such that, for all $1\leq i\leq k$,
$\Sigma_{i}:=\Sigma_{i-1}\cup \iota(P_{i})$ where $P_{i}$
is a 2-set contained in $\Sigma_{i-1} $ and $\iota(P_{i})$ is not
contained in $\Sigma_{i-1} $. 
We show by induction on $i$ that $\Sigma_{i}\subseteq Int(\Sigma)$ holds.

Clearly, if $i=0$ then  $\Sigma_0=\Sigma$ is contained in $Int(\Sigma)$.
So assume that  $\Sigma_{i}\subseteq Int(\Sigma)$ holds for all 
$1\leq i\leq r$, for some $1\leq r\leq k$, and that  
$\Sigma_r$ is obtained from $\Sigma_{r-1}$ by intersection of two 
splits $S_1, S_2\in \Sigma_{r-1}$. Since, 
by induction hypothesis,
$\Sigma_{r-1}\subseteq Int(\Sigma)$ it follows that 
$S_1$ and $S_2$ are contained in $Int(\Sigma)$. Since
$Int(\Sigma)$ is intersection-closed, $\iota(S_1,S_2)\subseteq 
Int(\Sigma)$ follows. Hence,
$\Sigma_r=\Sigma_{r-1}\cup \iota(S_1,S_2)\subseteq Int(\Sigma)$,
as required. By induction, it now follows that  
$\Sigma'\subseteq Int(\Sigma) $. Reversing the roles of $\Sigma'$ 
and $Int(\Sigma)$ in the previous argument implies that 
$Int(\Sigma)\subseteq \Sigma'$ holds too which implies
$\Sigma'= Int(\Sigma)$.
\end{proof}

 We remark in passing that similar arguments as the ones used in the 
proof of Lemma~\ref{lem:closure}
 also imply that the $\mathcal{I}$-intersection closed (set-inclusion)
minimal superset $\mathcal{I}(\Sigma)$
of a split system $\Sigma$ is also well-defined (and obviously
satisfies Properties (C1) -- (C3)).
We will refer to $\mathcal{I}(\Sigma)$ as 
{\em $\mathcal{I}$-intersection
closure of $\Sigma$}. 

We next turn our attention to the 
$\mathcal I$-intersection closure of a 
split systems induced by a 1-nested network.

\begin{proposition}\label{prop:intnet}
Suppose $N$ is a 1-nested network on $X$ and 
$S_1$ and $S_2$ are two incompatible splits contained in $\Sigma(N)$.
Then  $\iota(S_1,S_2)\subseteq \Sigma(N)$.
\end{proposition}
\begin{proof}
Note first that two splits $S$ and $S'$ 
induced by a 1-nested network are incompatible if and only if 
they are displayed by pairs of edges in the same cycle $C$ of $N$. 
For $i=1,2$, let $\{e_i,e'_i\}$ denote the edge set whose deletion
induces the split $S_i$. Then since $S_1$ and $S_2$ are incompatible,
we have $\{e_1,e'_1\}\cap\{e_2,e'_2\}=\emptyset$ and none of
the connected components of $N$ obtained by deleting 
$e_i$ and $e'_i$ contains both $e_j$ and $e'_j$, for all 
$i,j\in \{1,2\}$ distinct. Without loss of
generality, we may assume
that when starting at edge $e_1$ and moving clockwise through $C$ 
we first encounter $e_2$, then $e_1'$ and, finally $e_2'$
before returning to $e_1$. Then it is straight forward to see
that a split in $\iota(S_1,S_2)$ is displayed by one
of the edge sets $\{e_1,e_2\}$, $\{e_2,e_1'\}$, $\{e_1',e_2'\}$, 
and $\{e_2',e_1\}$. Thus, $\iota(S_1,S_2)\subseteq \Sigma(N)$.
\end{proof}

Combined with
the definition of the $\mathcal I$-intersection closure, we obtain

\begin{corollary}\label{cor:cii}
The following statements hold:
\begin{enumerate}
\item[(i)]
If $\Sigma$ is a circular split system for some circular ordering of $X$ then 
$\mathcal{I}(\Sigma)$ is also circular for that ordering.
\item[(ii)] If $N$ is a 1-nested network on $X$ then $\Sigma(N)$ is 
$\mathcal{I}$-intersection closed. Furthermore,  $N$
displays a split system $\Sigma$ on $X$ if and only if $N$ displays 
$\mathcal{I}(\Sigma)$.
\end{enumerate}
\end{corollary}



The next observation is almost trivial and is used in the proof of
Theorem~\ref{theo:equal}.

\begin{lemma}\label{lem:almost-trivial}
Suppose $x\in X$ and 
 $S_1$, $S_2$, and $S_3$ are three distinct splits of $X$
such that $S_3(x)\subseteq S_1(x)$, $S_3$ and $S_2$ 
are compatible and $S_1$ and $S_2$ are incompatible.
Then $S_3(x)\subseteq S_2(x)$ or $S_2(x)\subseteq S_3(x)$. 
\end{lemma}
\begin{proof}[Proof]
Since $S_2$ and $S_3$ are compatible either
$S_2(x)\subseteq S_3(x)$ or $S_3(x)\subseteq S_2(x)$ 
or  $\overline{S_3(x)}\subseteq S_2(x)$ must hold. If
$\overline{S_3(x)}\subseteq S_2(x)$ held then 
$\emptyset\not =\overline{S_1(x)}\cap \overline{S_2(x)}
\subseteq \overline{S_3(x)}\cap \overline{S_2(x)}=S_2(x)\cap \overline{S_2(x)}
=\emptyset$ follows
which is impossible.
  \end{proof}

For clarity of presentation we remark that for the
proof of Theorem~\ref{theo:equal}, we will assume 
that if a given split $S$ of a 1-nested network $N$
has multiplicity at least two in $\overline{\Sigma(N)}$
then $S$ is displayed by a cycle $C$ of $N$ (rather than by a cut-edge
of $N$). Furthermore, we denote the
split system of $X$ induced by a cycle $C$ of a 1-nested network 
$N$ on $X$ by $\Sigma(C)$. Clearly, $\Sigma(C)\subseteq \Sigma(N)$
holds.

\begin{theorem}\label{theo:equal}
Suppose $\Sigma$ is a split system on $X$ that contains all
trivial splits of $X$. Then
the following hold:
\begin{enumerate}
\item[(i)] There exists a 1-nested network $N$ on $X$ such that
$\Sigma=\Sigma(N)$ if and only if $\Sigma$ is circular and
 $\mathcal{I}$-intersection closed.
\item[(ii)] A maximal partially-resolved 1-nested network $N$ is a 
level-1 network if and only if
there exists no split of $X$ not contained in $\Sigma(N)$ 
that is compatible 
with every split in $\Sigma(N)$.
\end{enumerate}
\end{theorem}

\begin{proof}[Proof]
(i): Assume first that there exists a 1-nested network $N$
on $X$ such that $\Sigma=\Sigma(N)$. Then arguments similar to
the ones used in \cite[Theorem 2]{GBP12} to establish 
that the split system induced 
by a level-1 network is circular imply  that
$\Sigma(N)$ is circular. Hence, $\Sigma$ must be circular. That $\Sigma$ is 
$\mathcal{I}$-intersection closed follows by Corollary~\ref{cor:cii}(ii).

Conversely, assume that $\Sigma$ is circular and 
 $\mathcal{I}$-intersection closed.
Then there clearly exists a 1-nested network $N$ 
such that $\Sigma \subseteq \Sigma(N)$. Let $N$ be such a network 
such that  $|\Sigma(N)|$ is minimal among all 1-nested networks on $X$  
satisfying this 
set inclusion. Without loss of generality, we may assume that 
$N$ is maximal partially-resolved.
We show that, in fact,  $\Sigma= \Sigma(N)$ holds. 
Assume for contradiction that this is not the case, that is, there
exists a split $S_0 \in \Sigma(N) - \Sigma$. Since, by definition
of a phylogenetic network,  
$\Sigma(N)$ contains all trivial splits of $X$ it follows that
$S_0$ cannot be a trivial split of $X$. Also and in view of the
remark preceding the statement of the theorem, $S_0$ is induced 
by either (a) deleting a cut-edge  $e=\{u,v\}$ of $N$ and
neither $u$ nor $v$ are contained in a cycle of $N$
or (b) deleting two distinct
edges of the same cycle of $N$. 

Assume first that Case (a) holds.
Then collapsing $e$ results in a 1-nested network $N'$ on $X$ for which 
$\Sigma \subseteq \Sigma(N')$ holds. But then
 $|\Sigma(N')| < |\Sigma(N)|$ which is impossible in view of 
the choice of $N$.
 Thus, Case (b) must hold, that is, $S_0$ is induced by deleting two distinct
edges $e=\{u,v\}$ and $e'=\{u',v'\}$ of the same cycle $C$ of $N$. 
Let $x$ and $y$ be two elements of $X$ for which there exists a path from
$u$ and $v$, respectively, which does not cross an edge of $C$.
Consider the sets
$\Sigma_x:=\{S \in \Sigma \cap \Sigma(C): S(x) \subseteq S_0(x) \}$, and 
$\Sigma_y:=\{S \in \Sigma \cap \Sigma(C): S(y) \subseteq S_0(y)\}$. 
If $\Sigma_x$ is non empty then choose some 
$S_x\in \Sigma_x$  such that $|S_x(x)|$ is maximal among the splits 
contained in $\Sigma_x$. Similarly, define the split $S_y$ for 
$\Sigma_y$ if $\Sigma_y$ is non-empty. Otherwise let  $S_x$ be the
m-split of $C$ such that $S_x(x)\subseteq  S_0(x)$. Similarly, let 
 $S_y$ be the
m-split of $C$ such that $S_y(y)\subseteq  S_0(y)$ in case $\Sigma_y$
is empty. Then 
Corollary~\ref{cor:cii}(ii) implies that the split 
$$
S^*=S_x(x) \cup S_y(y) |\overline{ S_x(x)} \cap \overline{S_y(y)}
$$ 
is contained in $\Sigma(N)$ (see Figure~\ref{fig:redproc}(i)
for an illustration).
\begin{figure}[h]
\begin{center}
\includegraphics[scale=0.8]{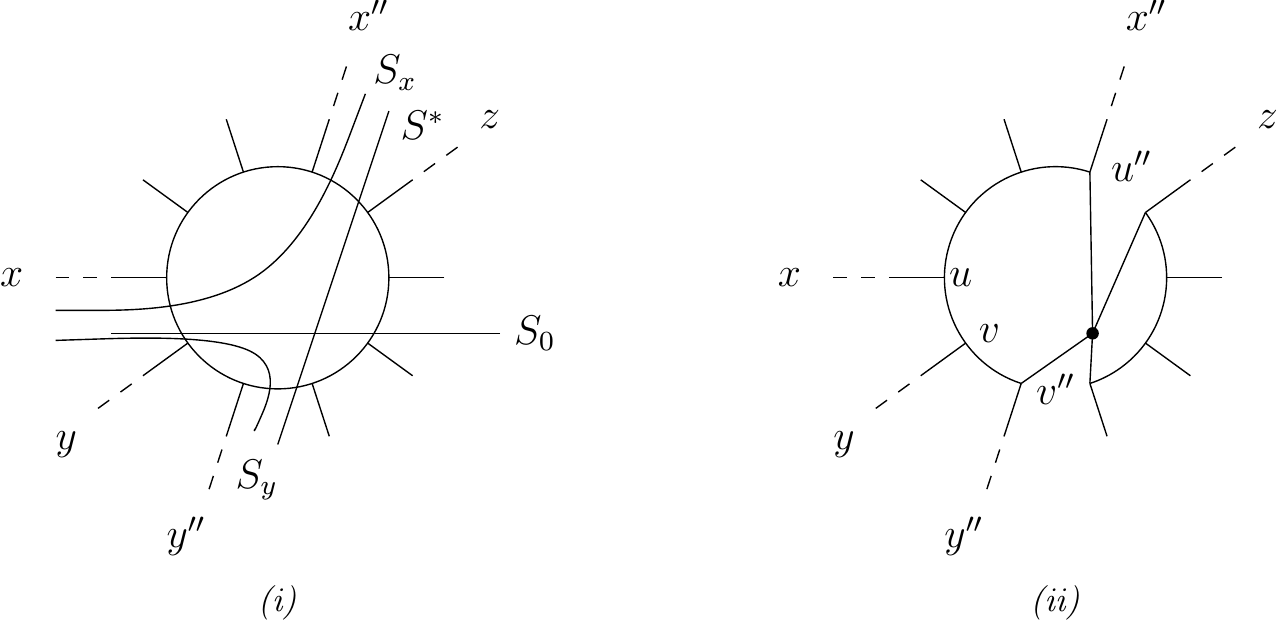}
\caption{\label{fig:redproc} (i) An illustration of the reduction 
process considered
in the proof of Case~(a) of Theorem~\ref{theo:equal}.
(ii) Again for that theorem, 
the graph $G'$ obtained from $N$ by adding subdivision vertices 
$r$ and $r'$. }
\end{center}
\end{figure}

We next show that $S^*$ is compatible with every split in $\Sigma$. 
To this end we first claim that every split 
 $S'\in  \Sigma$ that is incompatible with $S^*$
 must be compatible with at least one of $S_x$ and $S_y$. 
To see this, let $S'\in \Sigma$ such that
$S'$ and $S^*$ are incompatible. Then $S'$ must be displayed by $C$.
For contradiction, assume that $S'$ is incompatible with
both of $S_x$ and $S_y$.  Let $z\in X$ such that
$S^*(x)\not=S^*(z)$ and let $u''\in V(C)$ such
that $S_x(x)$ is the interval $[u,u'']$. Choose some element
$x''\in X$ such that there exists a path from  $x''$ to $u''$
that does not cross an edge contained in $C$. Similarly, 
let $v''\in V(C)$ such
that $S_y(y)$ is the interval $[v'',v]$. Choose some element
$y''\in X$ such that there exists a path from  $y'$ to $v''$
that does not cross an edge contained in $C$. Then  
since $S'$ is incompatible with $S_x$ and $S_y$ and displayed by
$C$ it follows that $S'(x'')= S'(y'')=S'(z)$. Hence,  
$S^*(z)\subseteq S'(z).$ 
But then  $S^* $ and $S'$ are not incompatible which is impossible. 
Thus $S'$ cannot be incompatible with both of 
$S_x$ and $S_y$, as claimed.

To see that $S^*$ is compatible with every split in $\Sigma$, we may,
in view of the above claim, assume
without loss of generality that $S'$ is compatible with  
$S_x$. Then Lemma~\ref{lem:almost-trivial} applied to 
$S'$, $S^*$, and $S_x$ implies
$S_x(x)\subsetneq S'(x)$ or $S'(x)\subsetneq S_x(x)$. 
If  $S'(x)\subsetneq S_x(x)$ held then
$\emptyset\not=S'(x)\cap \overline{S^*(x)}\subseteq 
S_x(x)\cap \overline{S^*(x)}=\emptyset$ follows which is
impossible. Hence, $S_x(x) \subsetneq S'(x)$. 
We distinguish between the cases that ($\alpha$) $S_y$ and $S'$
are compatible and ($\beta$) that they are incompatible.

Case ($\alpha$): Since  $S_y$ and $S'$
are compatible, similar arguments as above imply that 
$S_y(y) \subsetneq S'(y)$. Then the definition of 
$S^*$ combined with the assumption that $S'$ and $S^*$ 
are incompatible implies that $S'(x)\not=S'(y)$. But then
$S'$ and $S_0$ must be compatible, and so,  $S'(x)\subseteq S_0(x)$
or $S_0(x)\subseteq S'(x)$ must hold. If $S'(x)\subseteq S_0(x)$ held 
then $S'\in \Sigma_x$ which is impossible in view
of the choice of $S_x$ as  $S_x(x) \subsetneq S'(x)$.
Thus,   $S_0(x)\subseteq S'(x)$ must hold. But then
$S_y(y)\subsetneq S'(y)\subseteq S_0(y)$ and so
$S'\in \Sigma_y$ which is impossible in view of the
choice of $S_y$. Thus, Case~($\beta$) must hold.

Case ($\beta$): Since  $S_y$ and $S'$ are incompatible
the split
$$
S''=S'(x)\cap \overline{S_y(y)}|\overline{S'(x)}\cup S_y(y)
$$
is contained in $\Sigma$ because $\Sigma$ is
$\mathcal I$-intersection closed and clearly displayed by $C$.
Note that $x\in \overline{S''(y)}$
and so $S''(x)=\overline{S''(y)}$ must hold. Moreover, since 
$S'$ and $S^*$ are incompatible we cannot have $S''(x)=S_x(x)$
as $S_x$ and $S^*$ are compatible.
But then $S_0$ and $S''$ cannot be compatible. Indeed, if $S_0$ and $S''$  
were compatible then since $y\in S_0(y)\cap S''(y)$, 
$x\in \overline{S_0(y)}\cap \overline{S''(y)}$,
 and, because of $S_x(x)\subsetneq S'(x)$, also
$ \overline{S_0(y)}\cap S''(y)=S_0(x)\cap S''(y)=S_0(x)\cap
(\overline{S'(x)}\cup S_y(y))\subseteq S_0(x)\cap 
\overline{S_x(x)}\not=\emptyset$
holds, it follows that
$\overline{S''(y)}\subseteq \overline{S_0(y)}$, as required. Hence, 
$S''(x)= \overline{S''(y)}\subseteq \overline{S_0(y)}=S_0(x)$
and so $S''\in \Sigma_x$ which is impossible in view of the choice of 
$S_x$  as $S_x(x)\not=S''(x)$ and $S_0(x)\not=S''(x)$.
Thus, $S_0$ and $S''$ must be incompatible. But this is also impossible 
since the interval on $C$ corresponding to $S''(x)$ contains the interval
$[x,z]$ which induces the split $S_0$ and so $S_0$ and $S''$ must be
compatible. This final contradiction completes that proof that 
$S^*$ is compatible for every split in $\Sigma$.

To conclude,
let $G$ be a new graph obtained from $N$ by adding a subdivision vertex
 $r$ and $r'$, respectively, to each of two edges whose deletion induces 
the split $S^*$ (see 
Figure~\ref{fig:redproc}(ii) for an illustration). Then, the graph $G'$ 
obtained from $G$ by identifying $r$ and $r'$ is again a 1-nested network
on $X$. By construction, $S_0 \in \hat{\Sigma}:=\{S \in \Sigma(N) : \text{ S is 
incompatible with } S^*\}$ clearly holds and so 
$\Sigma(G')=\Sigma(N)-\hat{\Sigma}
\subsetneq\Sigma(N)$. 
 Since, by the above, 
every split in $\Sigma$ is compatible  with $S^*$ it follows 
that $\Sigma\subseteq \Sigma(G')$. But this
impossible in view of the choice of $N$. Hence, the split $S_0$ cannot
exist and, thus, $\Sigma=\Sigma(N)$.

(ii) Suppose $N$ is a maximal partially-resolved 1-nested network. Assume
first that $N$ is a 
level-1 network on $X$ and, for contradiction, that there
exists some split $S$ of $X$ not contained in $\Sigma(N)$ that is compatible
with every split in $\Sigma(N)$. Then $S'$ cannot be a trivial split of
$X$. Let $N'$ be the graph obtained from $N$ by deleting
from each cycle of $N$ one of its edges and suppressing resulting degree
two vertices. Clearly $N'$ is a phylogenetic tree on $X$ and since every
non-leaf vertex of $N$ had degree three every such vertex in $N'$ must 
also have degree three. Hence, $\Sigma(N')$ is a maximal compatible
split system on $X$. Since $S$ is compatible with every split of $\Sigma(N)$
and $\Sigma(N')\subseteq  \Sigma(N)$ it follows that 
$\Sigma(N')\cup \{S\}$ is also compatible which is impossible in view of the
maximality of $\Sigma(N')$.

Conversely, assume that there exists no split of $X$ not
contained in $\Sigma(N)$ that is compatible with every split in $\Sigma(N)$.
Then if $N$ is not level-1 it  contains a vertex $v$ of degree $k \geq 4$, that 
does not belong to a cycle of $N$.
Let $X_1,\ldots,X_k$ be the partition of $X$ obtained by deletion of $v$
(suppressing incident edges). 
Then there exist $i,j\in \{1,\ldots, k\}$ distinct, say $i=1$ and $j=2$,
such that the split $S:=X_1 \cup X_2| \bigcup_{i=3}^{k} X_i$ is 
compatible with every split in $N$. Since $S$ does not belong to 
$\Sigma(N)$ this is impossible.
\end{proof}

Note that Theorem~\ref{theo:equal} allows one to decide 
if for a split system $\Sigma$ of $ X$ there exists a 1-nested network $N$ 
on $X$ such that $\Sigma=\Sigma(N)$. However it does not
provide a tool for how to construct such a network. The provision of
such a tool is the purpose of the next two sections. 

\section{Optimality and the analogue of the Split Equivalence 
Theorem}\label{sec:cir-ord}

In this section, we turn our attention towards constructing 
1-nested networks from circular split systems. In particular, we
show that for any circular split system $\Sigma$ on $X$ it is possible to 
construct a, in a well-defined sense, optimal
1-nested network on $X$ in $\mathcal O(|X|^2+|\Sigma|^2)$ time
(Theorem~\ref{theo:inclu}). Central to our proof is Theorem~\ref{theo:ordun}
in which we characterize circular split systems whose $\mathcal I$-intersection
closure is (set-inclusion) maximal in terms of their so called incompatibility
graphs. As a consequence, 
we obtain the analog of the ``Splits-Equivalence
Theorem'' (see Section~\ref{sec:intro}) for 1-nested networks 
(Corollary~\ref{cor:equivalence}).
For phylogenetic trees this theorem is fundamental and characterizes
split systems $\Sigma$ for which there exists a, up to
isomorphism, unique phylogenetic tree
$T$ for which $\Sigma=\Sigma(T)$ holds.

We start with introducing
some more terminology. Suppose $\Sigma$ is a circular split 
system on $X$. Then 
we say that $\Sigma$ is \emph{maximal circular} if for all 
split system $\Sigma'$ 
that contain $\Sigma$, we have $\Sigma=\Sigma'$.
As the next result illustrates, maximal circular split systems of $X$ and
1-nested networks on $X$ are closely related.

\begin{lemma}\label{lem:maxcirc}
A split system $\Sigma$ on $X$ is maximal circular if and only if 
there exists a simple
level-1 network $N$ on $X$  such that $\Sigma = \Sigma(N)$.
\end{lemma}

\begin{proof}
Let $\Sigma$ be a split system on $X$. Assume first that
$\Sigma$ is maximal circular. Then,
exists a simple level-1 network $N$ on $X$ such that 
$\Sigma \subseteq \Sigma(N)$. Since $\Sigma(N)$ is clearly 
a circular split system on $X$ the maximality of $\Sigma$ implies
$\Sigma = \Sigma(N)$.

Conversely, assume that $N$ is a simple level-1 network such
that $\Sigma = \Sigma(N)$. 
Then since $\Sigma(N)$ is a circular split system on $X$ so is $\Sigma$.
Assume for contradiction that $\Sigma$ is not maximal circular,
that is, there exists a split $S=A|\bar A\in \Sigma$
that is not contained in $\Sigma(N)$. Then $A$
and $\bar A$ are both intervals on the circular ordering 
of $X$ induced by $\Sigma(N)$. Hence, 
$S$ is induced by a minimal cut
of $N$ and, so, $S\in\Sigma(N)$ which is impossible.
\end{proof}

Note that since a maximal circular split system on $X$ must necessarily contain
all 2-splits of $X$ obtainable as a minimal 
cuts in the associated simple level-1 network on $X$,
it follows that that ordering of $X$ is unique. The next result suggests that 
systems of such splits suffice to generate a maximal circular 
split system. To state it,
suppose $x_1, ..., x_{n-1}, x_n, x_{n+1}=x_1$
is a circular ordering of $X$ and put
$\Sigma_d:=\{\{x_i,x_{i+1}\}| X-\{x_i,x_{i+1}\}\, :\, 1\leq i\leq n\}$.
Clearly, $\Sigma_d$ is a circular split system on $X$.

In view of Lemma~\ref{lem:maxcirc}, we say that a circular ordering
{\em displays} a split system $\Sigma$ if $\Sigma$ is displayed
by the simple level-1 network associated to $\Sigma$.

\begin{lemma}\label{lem:2sp}
Suppose $\sigma: x_1, ..., x_{n-1}, x_n, x_{n+1}=x_1$ is 
a circular ordering
of $X$. Then  
$\mathcal I(\Sigma_d)$ is a maximal circular split system on $X$.
\end{lemma}

\begin{proof}[Proof]
Since the result is trivial for $n=3$, we may assume without 
loss of generality that $n \geq 4$.
We proceed by induction on the size $1 \leq l \leq \frac{n}{2}$ 
of a split $S$ displayed
by $\sigma$. Suppose first that $l=1$. Then there
exists some $i\in \{1,\ldots, n\}$ such that $S=x_i|X-\{x_i\}$. Clearly,
$S_1=\{x_i,x_{i-1}\}|X-\{x_i,x_{i-1}\}$
and $S_2=\{x_i,x_{i+1}\}|X-\{x_i,x_{i+1}\}$ 
are contained in $\Sigma_d$ and incompatible. Hence,
$S=S_1(x_i)\cap S_2(x_i)|X-(S_1(x_i)\cap S_2(x_i))
\in \iota(S_1,S_2)\subseteq \mathcal I(\Sigma_d)$. 

Now assume that $l\geq 2$ and
that all splits of $X$ displayed by  $\sigma$ of size at most $l-1$ 
are contained
in $\mathcal I(\Sigma_d)$. Since $S$ is displayed by $\sigma$ there 
exists some $i\in \{1,\ldots, n\}$ such that
$S=[x_i, x_{i+l-1}]|X-[x_i, x_{i+l-1}]$.
Without loss of generality we may assume that $i=1$.
Then $S=[x_1,x_l]|X-[x_1,x_l]$.
Consider the splits $S_1=[x_1, x_{l-1}]|X-[x_1, x_{l-1}]$
and $S_2=\{x_{l-1},x_l\}|X-\{x_{l-1}, x_l\}$ displayed by $\sigma$. 
By induction, $S_1,S_2\in \mathcal I(\Sigma_d)$ since
the size of $S_2$ is two  and that of $S_1$ is at most $l-1$.
Furthermore, $S_1$ and $S_2$ are incompatible. 
Since
$S=S_1(x_{l-1})\cup S_2(x_{l-1})|X-(S_1(x_{l-1})\cup S_2(x_{l-1})
\in \iota(S_1,S_2)\subseteq \mathcal I(\Sigma_d)$, the lemma follows.
\end{proof}

We next employ Lemma~\ref{lem:2sp} to obtain
a sufficient condition on a circular split system $\Sigma$
for $\mathcal I(\Sigma)$ to be 
maximal circular. Central
to this is the concept of the {\em incompatibility graph} $Incomp(\Sigma)$
of a split system $\Sigma$. For $\Sigma$
a split system on $X$ the vertex set of that graph is 
$\Sigma$  and any two distinct splits of $\Sigma$
are joined by an edge in $Incomp(\Sigma)$ if they are incompatible.
We denote the
set of connected components of $Incomp(\Sigma)$ by $\pi_0(\Sigma)$
and, by abuse of terminology, refer to the vertex set of an element in 
$\pi_0(\Sigma)$ as a {\em connected component of $Incomp(\Sigma)$}. 
 For example, $Incomp(\Sigma_d)$ is a cycle of length $|\Sigma_d|$
whenever $n\geq 5$.
Furthermore,  $\Sigma$ is compatible if and only if $|\Sigma_0|=1$
holds for all $\Sigma_0\in \pi_0(\Sigma)$.
 
We next clarify the relationship between the incompatibility graph and
$\mathcal{I}$-intersection closure of a split system.

\begin{lemma}\label{lem:incomp-con-com}
Suppose $\Sigma$ is a split system on $X$.  
Then there cannot exist two distinct connected components
$\Sigma_1, \Sigma_2\in \pi_0(\Sigma)$
and splits $S_1\in \mathcal{I}(\Sigma_1)$
and $S_2\in \mathcal{I}(\Sigma_2)$ such that $S_1$ and $S_2$ 
are incompatible.
\end{lemma}
\begin{proof}
Assume for contradiction that 
there exist two connected components $\Sigma_1,\Sigma_2\in \pi_0(\Sigma)$
and splits $S_1\in  \mathcal{I}(\Sigma_1)$
and $S_2\in  \mathcal{I}(\Sigma_2)$ such that $S_1$ and
$S_2$ are incompatible.
Then
$S_1 \in \Sigma_1$ and $S_2 \in \Sigma_2$ cannot both hold as otherwise 
$\Sigma_1=\Sigma_2$. 
Assume without loss of generality that $S_1 \notin \Sigma_1$. 
Let $\Sigma^0:=\Sigma_1  \subsetneq \Sigma^1 \subsetneq \ldots \subsetneq 
\Sigma^k:= \mathcal{I}(\Sigma_1)$, $k\geq 1$ be a finite 
sequence such that, for all $1\leq i\leq k$, a split in 
$\Sigma^i$ either belongs to $\Sigma^{i-1}$ or is an 
$\mathcal{I}$-intersection between two splits $S,S'\in\Sigma^{i-1}$
and $\iota(S,S')\not\subseteq \Sigma^{i-1}$. 
Then, there exists some $i^*>0$ such that 
$S_1 \in \Sigma^{i^*} - \Sigma^{i^*-1}$. After possibly renaming $S_1$, we
may assume without loss of generality, that
$i^*$ is such that for all $1\leq i\leq i^*-1$ there
exists no split in $\Sigma^{i}$ that is incompatible with 
$S_2$. Hence, there must exist two splits 
$S$ and $S'$ in $\Sigma^{i^*-1}$ distinct such that $S_1 \in \iota(S,S')$. 
Since $S_2$ and $S_1$ are incompatible, it follows that 
$S_2$ is incompatible 
with one of  $S$ and $S'$, which is impossible by the choice of $i^*$.
\end{proof}

Armed with this result, we next relate for a split system
$\Sigma$ the sets $\pi_0(\mathcal{I}(\Sigma))$ and $\pi_0(\Sigma)$.

\begin{lemma}\label{lem:intinc}
Suppose $\Sigma$ is a split system on $X$. Then the following hold
\begin{enumerate}
\item[(i)] $\mathcal{I}(\Sigma)=\bigcup_{\Sigma_0 \in \pi_0(\Sigma)} 
\mathcal{I}(\Sigma_0)$.
\item[(ii)] 
$\pi_0(\mathcal{I}(\Sigma_0)) \subseteq \pi_0(\mathcal{I}(\Sigma))$,  
for all $\Sigma_0\in\pi_0(\Sigma)$. 
In particular, $\pi_0(\mathcal{I}(\Sigma))=\bigcup_{\Sigma_0 \in \pi_0(\Sigma)} 
\pi_0(\mathcal{I}(\Sigma_0))$.
\end{enumerate}
\end{lemma}

\begin{proof}
(i) Let $\Sigma_0\in \pi_0(\Sigma)$ and  
put $\mathcal A:=\bigcup_{\Sigma' \in \pi_0(\Sigma)} \mathcal{I}(\Sigma')$.
Note that since $\Sigma=\bigcup_{\Sigma' \in \pi_0(\Sigma)} \Sigma'$, 
we trivially have $\Sigma \subseteq \mathcal A \subseteq \mathcal{I}(\Sigma)$.  
To see that $\mathcal{I}(\Sigma)\subseteq \mathcal A$ note that
Lemma~\ref{lem:incomp-con-com} implies that any two
incompatible splits in $\mathcal A$ must be
contained in the same connected component so must be their 
$\mathcal I$- intersection. Thus, 
$\mathcal A$ is $\mathcal{I}$-intersection closed. 
Since $\Sigma\subseteq \mathcal  A$ we also have
 $\mathcal{I}(\Sigma) \subseteq \mathcal{I}( \mathcal  A)
= \mathcal A$ and, so,  
$\mathcal A=\mathcal{I} (\Sigma)$ follows.

(ii) Suppose $\Sigma_0 \in \pi_0(\Sigma)$ and 
let $\overline{\Sigma}_0\in\pi_0(\mathcal{I}(\Sigma_0))$. 
To establish that 
$\overline{\Sigma}_0\in\pi_0(\mathcal{I}(\Sigma))$
note that since $\overline{\Sigma}_0$ is connected in $\mathcal{I}(\Sigma_0)$
it also is connected in  $\mathcal{I}(\Sigma)$. 
Hence, it suffices to show that every split
in $\overline{\Sigma}_0$ is compatible with every split in 
$\mathcal{I}(\Sigma)-\overline{\Sigma}_0$. Suppose 
$S_1\in \overline{\Sigma}_0$ and $S_2\in 
\mathcal{I}(\Sigma)-\overline{\Sigma}_0=
(\mathcal{I}(\Sigma)-\mathcal{I}(\Sigma_0)) \cup (\mathcal{I}(\Sigma_0)-
\overline{\Sigma}_0)$. If $S_2\in \mathcal{I}(\Sigma_0)-
\overline{\Sigma}_0$ then, by definition, $S_1$ and $S_2$ are compatible.
So assume that $S_2\in \mathcal{I}(\Sigma)-\mathcal{I}(\Sigma_0)$.
Then Lemma~\ref{lem:intinc}(i) implies that $S_2$ is compatible
with every split in $\mathcal{I}(\Sigma_0)$ and thus 
with $S_1$ as $\overline{\Sigma}_0\subseteq \mathcal{I}(\Sigma_0)$.
\end{proof}

To establish the next result which is central to Theorem~\ref{theo:inclu}, 
we require a further notation.
Suppose $\Sigma$ is a split system on $X$. Then we denote by
 $\Sigma^-$ the split system obtained from $\Sigma$
by deleting all trivial splits on $X$.

\begin{theorem}\label{ordun}\label{theo:ordun}
Let $\Sigma$ be a circular split system on $X$. Then $\mathcal I(\Sigma)$ is a 
maximal circular split system on $X$ if and only if the  
following two conditions hold:\\
(i) for all $x, y \in X$ distinct, there exists some $S \in \Sigma^-$ 
such that $S(x) \neq S(y)$,\\
(ii) $Incomp(\Sigma^-)$ is connected.\\
Moreover, if (i) and (ii) hold then
there exists an unique, up to isomorphism and partial-resolution,
simple 1-nested network $N$ on $X$ 
such that $\Sigma \subseteq \Sigma(N)$.
\end{theorem}

\begin{proof}[Proof]
Let $x_1, ..., x_n, x_{n+1}=x_1$ denote an underlying
circular ordering of $X$ 
for  $\Sigma$. Assume first that \emph{i)} and \emph{ii)} hold. 
We first show that $\mathcal I (\Sigma^-)$
is maximal circular. To this end, it suffices to show that 
$\Sigma_d\subseteq\mathcal{I}(\Sigma^-)$ since this implies that
$\mathcal{I}(\Sigma_d) \subseteq \mathcal{I}(\mathcal{I}(\Sigma^-))
 \subseteq \mathcal{I}(\mathcal{I}(\Sigma))=
\mathcal{I}(\Sigma)$. Combined with the fact that,
 in view of Lemma~\ref{lem:2sp},
$\mathcal{I}(\Sigma_d)$ is maximal circular,
it follows that $\mathcal I(\Sigma_d)=\mathcal{I}(\Sigma^-)=
\mathcal{I}(\Sigma)$.
Hence, $\mathcal{I}(\Sigma)$ is maximal circular.

Assume for contradiction that there
exists some $i\in\{1,\ldots, n\}$ such that the split 
$S^*=x_ix_{i+1}|x_{i+2},\ldots x_{i-1}$ of $\Sigma_2$ is not contained in $\mathcal I( \Sigma)$. 
Then, by assumption, there exist two splits  
$S$ and $S'$ in $\Sigma$ 
such that $S(x_i) \neq S(x_{i-1})$ and  $S'(x_{i+1}) \neq S'(x_{i+2})$.
Let $P_{SS'}$ denote a shortest path in
$Incomp(\Sigma)$ joining $S$ and $S'$. Without loss of generality,
let $S$ and $S'$ be such that the path $P_{SS'}$ is a short as
possible. Let $S_0=S, S_1,..., S_k=S'$ denote that path. The next
lemma is central to the proof

\begin{lemma}\label{lem:central-char-max-circular} 
For all $0 \leq j \leq k$, we have
$S_j(x_i)=S_j(x_{i+1})$. 
\end{lemma}
\begin{proof}[Proof]
First observe that 
$S_j(x_i) = S_j(x_{i-1})$ and 
$S_j(x_{i+1})= S_j(x_{i+2})$ must hold for all 
$0 < j < k$. Indeed, if
there existed some $j\in \{1,\ldots, k-1\}$ such
that $S_j(x_i) \not= S_j(x_{i-1})$ then the path 
$S_j, S_{j+1},\ldots, S_k$ would be shorter than
 $P_{SS'}$ in contradiction to the choice of $S$ and $S'$. 
Similar arguments also imply that $S_j(x_{i+1})= S_j(x_{i+2})$
holds  for all $j\in \{1,\ldots, k-1\}$.

Assume for contradiction that there exists 
$0 \leq j \leq k$ such that $S_j(x_i) \neq S_j(x_{i+1})$. 
Without loss of generality, we may assume
that for all $0\leq l\leq j-1$ we have that
$S_l(x_i)=S_l(x_{i+1})$. Then since
a trivial split cannot be incompatible with any other split on $X$ we
cannot have $j\in\{0,k\}$. Thus, the splits 
$S_{j-1}$ and $S_{j+1}$ must exist. Note that they 
cannot be incompatible, since otherwise the path from $S$ to $S'$
obtained by deleting $S_j$ from $P_{SS'}$ is shorter than 
$P_{SS'}$ which is impossible.
So $S_{j-1}$ and $S_{j+1}$ must be compatible. 
Clearly, $x_i\in S_{j+1}(x_i)\cap S_{j-1}(x_i)$. We next
establish that $\overline{S_{j+1}(x_i)}\cap \overline{S_{j-1}(x_i)}=\emptyset$
cannot hold implying that either
 $S_{j+1}(x_i)\cap \overline{S_{j-1}(x_i)}=\emptyset$ or
$\overline{S_{j+1}(x_i)}\cap S_{j-1}(x_i)=\emptyset$.

Indeed, let $q \in \{1, \ldots, n\}$ such that $S_j=x_{i+1} \ldots x_{q}|x_{q+1} \ldots x_{i}$. We claim that $x_q \in \overline{S_{j+1}(x_i)} \cap \overline{S_{j+1}(x_i)}$. Assume by contradiction that $x_q \in S_{j-1}(x_i)$. Then $S_{j-1}(x_i)$ is an interval of $X$ containing $\{x_i,x_q\}$ and, so, either 
$S_{j-1}(x_i) \supseteq [x_i, x_q] \supset S_j(x_{i+1})$ or 
$S_{j-1}(x_i) \supseteq [x_q, x_i] \supset S_j(x_{i})$.
But both are impossible in view of the fact that $S_{j-1}$ and $S_j$ 
are incompatible. 

Now assume that $S_{j+1}(x_i)\cap \overline{S_{j-1}(x_i)}=\emptyset$,
that is, $S_{j+1}(x_i)\subseteq S_{j-1}(x_i)$.
We postulate that then $S_{j+1}(x_i)\subseteq S_0(x_i)$ must hold which
is impossible since $x_{i-1}\in S_{j+1}(x_i)$ and $S_0(x_i)\not=S_0(x_{i-1})$.
Indeed, the choice of $S$ and $S'$ implies that $S_{j+1}$
and $S_l$ must be compatible, for all $0\leq l\leq j-2$. By 
Lemma~\ref{lem:almost-trivial} applied to $S_{j-1}$, $S_{j-2}$, 
and $S_{j+1}$ it follows
that $S_{j+1}(x_i)\subseteq S_{j-2}(x_i)$ or $S_{j-2}(x_i)\subseteq S_{j+1}(x_i)$.
In the latter case we obtain $S_{j-2}(x_i)\subseteq S_{j-1}(x_i)$ 
which is impossible since $S_{j-1}$ and $S_{j-2}$ are incompatible.
Thus, $S_{j+1}(x_i)\subseteq S_{j-2}(x_i)$. Repeated application
of this argument implies that, for all $0\leq l\leq j-2$
we have $S_{j+1}(x_i)\subseteq S_l(x_i)$, as required.

Finally, assume that $S_{j-1}(x_i)\cap \overline{S_{j+1}(x_i)}=\emptyset$,
that is, $S_{j-1}(x_i)\subseteq S_{j+1}(x_i)$. Then similar arguments
as in the previous case imply that  $S_{j-1}(x_i)\subseteq S_k(x_i)$. 
But this is impossible since $x_{i+1}, x_{i+2} \in S_{j-1}(x_i)$
and $S_k(x_i)=S_k(x_{i+1})\not=S_k(x_{i+2})$.
Thus, $S_j(x_i)=S_j(x_{i+1})$ must hold for all $0 \leq j \leq k$ 
which concludes the proof of Lemma~\ref{lem:central-char-max-circular}. 
\end{proof}

Continuing with the proof of Theorem~\ref{theo:ordun}, 
we claim that the splits 
$$
T_j:=T_{j-1}(x_i) \cap S_j(x_i)|
\overline{T_{j-1}(x_i)} \cup \overline{S_j(x_i)}
$$ 
where $j\in \{1,\ldots,k\}$ and $T_0:=S_0$
are contained in $\mathcal{I}(\Sigma)$. 
Assume for contradiction that there exists some
$j\in \{0,\ldots,k\}$ such that $T_j\not\in \mathcal{I}(\Sigma)$.
Then $j\not=0$ because $S\in \mathcal{I}(\Sigma)$, and $j \not= 1$ 
since $T_1\in\iota(S,S_1)$ and $S,S_1\in \Sigma$. 
Without loss of generality, we may assume that $j$ is such that
for all $1\leq l\leq j-1$ we have that $T_l\in \mathcal{I}(\Sigma)$.
Then $ T_{j-1}$  and $ S_j$ cannot be incompatible and so
 $T_{j-1}(x_i) \subseteq S_j(x_i)$, or 
$S_j(x_i) \subseteq T_{j-1}(x_i)$, or $\overline{S_j}(x_i) \subseteq T_{j-1}(x_i)$
must hold. But $S_j(x_i) \subseteq T_{j-1}(x_i)$ cannot hold
since then $\overline{S_{j-1}(x_i)}\subseteq
\overline{T_{j-2}(x_i)} \cup \overline{S_{j-1}(x_i)}
= \overline{T_{j-1}(x_i)}\subseteq \overline{S_j(x_i)}$
 which is impossible as $S_{j-1}$ and $S_j$ are incompatible. 
Also, $\overline{S_j}(x_i) \subseteq T_{j-1}(x_i)$ cannot hold
since then 
$\overline{S_j}(x_i) \subseteq T_{j-1}(x_i)=T_{j-2}(x_i)\cap S_{j-1}(x_i)
\subseteq S_{j-1}(x_i)$ which is again impossible as  $S_{j-1}$ and $S_j$
are incompatible. Thus,  $T_{j-1}(x_i) \subseteq S_j(x_i)$ 
and so $T_j(x_i)=T_{j-1}(x_i)$. Consequently,
$T_j = T_{j-1}\in \mathcal{I}(\Sigma)$ which is also impossible
and therefore proves the claim.
Thus, $T_j\in \mathcal{I}(\Sigma)$, for all  $0\leq j\leq k$.
Combined with Lemma~\ref{lem:central-char-max-circular}
it follows that, for all $0 \leq j \leq k$,
we also have $T_j(x_i)=T_j(x_{i+1})$. Consequently,
$\{x_i,x_{i+1}\} \subseteq T_k(x_i)$. Combined with the
facts that $T_k(x_i)$ is an interval on $X$ 
and $x_{i-1} \notin S_0(x_i)$, and similarly,
$x_{i+2} \notin S_k(x_i)$ it follows that 
 $\{x_i,x_{i+1}\} = T_k(x_i)$. Hence, 
 $S^*=T_k \in \mathcal{I}(\Sigma)$, which is impossible.
Thus, $\Sigma_d\subseteq \mathcal I( \Sigma^-)$ and so
$\mathcal I(\Sigma_d) =\mathcal I( \Sigma^-)$.

Conversely, assume that $\mathcal I(\Sigma)$ is maximal circular.
Then $\mathcal I(\Sigma)$ clearly satisfies Properties (i) and (ii)
that is, for all $x,y\in X$ distinct there exists some
$S\in \mathcal I(\Sigma)^-$ such that $S(x)\not=S(y)$
and $Incomp(\mathcal I(\Sigma)^-)$ is connected.
We need to show that $\Sigma$ also satisfies Properties (i) and (ii).
Assume for contradiction that $\Sigma$ does not satisfy Property~(i).
Then there exist $x,y\in X$ distinct such that for 
all splits $S\in \Sigma^-$, we have $S(x)=S(y)$. Let $S\in \mathcal I(\Sigma)$
such that $S(x)\not=S(y)$ and let $S_1,S_2,\ldots,S_l=S$ denote a
sequence in $\mathcal I(\Sigma)$ such that $S_i\in \iota(S_{i-1},S_{i-2})$,
for all $3\leq i\leq l$. Without loss of generality we may assume that
$l$ is such that  $S_i(x)=S_i(y)$, for all  $3\leq i\leq l-1$. Then
$S_j(x)=S_j(y)$, for all $j\in \{l-1,l-2\}$ and thus $S(x)=S(y)$ which
is impossible.

Next, assume for contradiction that $\Sigma$ does not satisfy Property~(ii). 
Let $\Sigma_1$ and $\Sigma_2$ denote two disjoint connected components  
of $Incomp(\Sigma^-)$.  For $i=1,2$, let 
$\overline{\Sigma_i}\in \pi_0(\mathcal I(\Sigma_i)^-)$
such that $\Sigma_i\subseteq \overline{\Sigma_i}$.
 Then, $2\leq |\Sigma_i|\leq |\overline{\Sigma_i}|$, for all $i=1,2$.
Combined with Lemma~\ref{lem:intinc}(ii), we obtain 
$\overline{\Sigma_1},\overline{\Sigma_2}\in \pi_0(\mathcal{I}(\Sigma)^-)$. 
Since  $Incomp(\mathcal I(\Sigma)^-)$ is connected, 
it follows for $i=1,2$ that 
$\overline{\Sigma_i}\subseteq \mathcal I(\Sigma_i)^-\subseteq 
\mathcal I(\Sigma)^- =\overline{\Sigma_i}$. Thus, 
$\mathcal I(\Sigma_1)^-=\mathcal I(\Sigma^-)=\mathcal I(\Sigma_2)^-$
and so the incompatibility graphs $Incomp(\mathcal I(\Sigma_1)^-)$,
$Incomp(\mathcal I(\Sigma)^-)$  and
$Incomp(\mathcal I(\Sigma_2)^-)$ all coincide.
Suppose $S\in \Sigma_1$ and $S'\in \Sigma_2$ and let $P$ denote a
shortest path in $Incomp(\mathcal I(\Sigma)^-)$ joining $S$ and $S'$.
Then there must exist incompatible splits $S$ and $S'$ in $P$
such that $S\in \Sigma_1\subseteq\mathcal I(\Sigma_1)^-$ and
 $S'\in \mathcal I(\Sigma_1)^-=\mathcal I(\Sigma_2)^-$ 
which  is impossible in view of Lemma~\ref{lem:incomp-con-com}.

The remainder of the theorem follows from the facts that, 
by Lemma~\ref{lem:2sp},
$\mathcal{I}(\Sigma)$ is maximal circular that, by 
 Lemma~\ref{lem:maxcirc}, there exists a simple level-1 network $N$ such
that $\mathcal{I}(\Sigma)=\Sigma(N)$, that by
Corollary~\ref{cor:cii}(ii), a 1-nested network displays 
displays $\mathcal{I}(\Sigma)$ if and only if it displays $\Sigma$,
and that the split system $\Sigma_d$ uniquely determines the 
underlying circular ordering of $X$.
\end{proof}

Armed with this characterization, we are now ready to establish 
Theorem~\ref{theo:inclu}

\begin{theorem}\label{theo:inclu}
Given a circular split system $\Sigma$ on $X$, it is possible to build, 
in time $O(n^2+|\Sigma|^2)$, a 1-nested network $N$ on $X$ such that 
$\Sigma \subseteq \Sigma(N)$ holds and $|\Sigma(N)|$ is minimal.
Furthermore, $N$ is unique up to isomorphism and partial-resolution. 
\end{theorem}

\begin{proof}[Proof]
Suppose $\Sigma$ is a circular split system on $X$. 
Put $\{V_1,\ldots,V_l\}=\pi_0(\Sigma)$.  
 Without loss of generality
we may assume that there exists some $j\in\{1,\ldots,l\}$ 
such that $|V_i|=1$ holds for all $1\leq i\leq j-1$ and
$|V_i|\geq 2$ for all $j\leq i\leq l$.
Since $Incomp(\Sigma)$ has $l-j+1$ 
connected components with at least two vertices there exist $l-j+1$
simple 1-nested networks $N_i$ such that $V_i\subseteq \Sigma(C_i)$
holds for the unique cycle $C_i$ of $N_i$. By Theorem~\ref{theo:ordun},
it follows for all $j\leq i\leq l$ that 
$\Sigma(C_i)=\mathcal I(V_i)$ and that $Q_i \subseteq \mathcal I(V_i)$, where 
$Q_i$ denotes the set of m-splits of $C_i$. 

We claim 
that the split system $\Sigma'$ on $X$ given by
$$
\Sigma'=\bigcup_{i=1}^{j-1}V_i\cup\bigcup_{i=j}^lQ_i\cup\bigcup_{x\in X}\{x|X-x\}
$$
is compatible.
Since $\Sigma$ is circular there exists a 1-nested network $N$ on $X$
such that $\Sigma\subseteq \Sigma(N)$. Without loss of generality,
we may assume that $N$ is such that $|\Sigma(N)|$ is minimal among 
such networks. For clarity of exposition, we may furthermore assume
that $N$ is maximal partially-resolved. Then for all $j\leq i\leq l$
there exists a cycle $Z_i$ in $N$ such that $V_i\subseteq \Sigma(Z_i)$.
In fact, $\mathcal I(V_i)= \Sigma(Z_i)$ must hold
for all such $i$. Combined with the minimality
of $\Sigma(N)$, it follows that there exists a one-to-one correspondence
between the cycles of $N$ and the set 
$\mathcal A:=\{\mathcal I(V_i)\,:\, j\leq i\leq l\}$
that maps a cycle  $Z$ of $N$ to
the split system $\Sigma_C\in\mathcal A$
such that for some $i^*\in \{j,\ldots,l\}$ we have 
$\Sigma_C=\mathcal I(V_{i^*})$ and $V_{i^*}\subseteq \Sigma(Z)$.
Furthermore, for all $1\leq i\leq j-1$ there exists a 
cut-edge $e_i$ of $N$ such that the split $S_{e_i}$ induced on $X$ by
deleting $e_i$ is the unique element in $V_i$.

Let $T(N)$ denote the phylogenetic tree on $X$ obtained
from $N$ by first shrinking every cycle $Z$ of $N$ to a vertex $v_Z$ and
then suppressing all resulting degree two vertices. Since 
this operation clearly preserves the splits in $Q_i$, $j\leq i\leq l$,
and also does not affect the cut-edges of $N$ (in the sense
that a cut edge of $T(N)$ might correspond to a path
in $N$  of length at most 3 involving a cut-edge of $N$ 
and one or two m-splits),  it follows that $\Sigma' =\Sigma(T(N))$.
Since any split system displayed by a phylogenetic tree is 
compatible the claim follows.

Since, in addition, $\Sigma'$ also contains all trivial splits on $X$, 
it follows by the ``Splits Equivalence Theorem'' 
(see Section~\ref{sec:intro}) that
there exists a unique (up to isomorphism)
phylogenetic tree $T$ on $X$ such $\Sigma(T)=\Sigma'$. Hence, 
$T(N)$ and $T$ must be 
isomorphic. But then reversing the aforementioned cycle-shrinking operation
that gave rise to $T(N)$ results in a 1-nested network $N'$ on $X$
such that $\Sigma(N)=\Sigma(N')$. Consequently, $N'$ and $N$
are isomorphic and so $\Sigma\subseteq \Sigma(N')$. Note that similar
arguments also imply that $N$ is unique up to partial-resolution and
isomorphism.


To see the remainder of the theorem, note first that 
a maximal circular split system $\Sigma$ on $X$ has $n(n-1)/2$ splits. 
Thus, finding
$Incomp(\Sigma)$ can be accomplished in $\mathcal O(|\Sigma|^2)$ time.
Combined with the facts that $X$ has at most $n$ cycles and any
binary unrooted phylogenetic tree on $X$ has  $2n-3$ cut-edges it follows 
that $N'$ can be constructed in $\mathcal O(n^2+|\Sigma|^2)$ time.
\end{proof}

In consequence of Theorems~\ref{theo:equal}
 and \ref{theo:inclu},  we obtain the 1-nested
analogue of the ``Splits Equivalence Theorem'' (see Section~\ref{sec:intro}).

\begin{corollary}\label{cor:equivalence}
Suppose $\Sigma$ is a split system on $X$ that contains all
trivial splits of $X$. Then there exist a 1-nested network 
$N$ on $X$ such that $\Sigma=\Sigma(N)$ if and only if  $\Sigma$ is circular and
 $\mathcal{I}$-intersection closed. Moreover, if such a network
$N$ exists then it is unique up to isomorphism and partial-resolution
and can be constructed in $\mathcal O(n^2+|\Sigma|^2)$ time.
\end{corollary}

As observed in Section~\ref{sec:prelim}, a 1-nested network also induces
a multi-set of splits. This raises the question of an 1-nested
analogue of the ``Splits Equivalence Theorem'' 
(see Section~\ref{sec:intro}) for such collections. We will settle 
this question elsewhere.

\section{Optimality and the Buneman graph}\label{sec:buneman}

In this section, we investigate the interplay between
the Buneman graph $G(\Sigma)$ of a circular split system $\Sigma$
and a 1-nested network displaying $\Sigma$. More precisely,
we first associate to a circular split system $\Sigma$
a certain subgraph of $G(\Sigma)$ 
which we obtain by replacing each block of
$G(\Sigma)$ by a structurally simpler graph which we call a
marguerite. As it turns out, marguerites hold the key for constructing
optimal 1-nested networks from circular split systems.
We start with defining the Buneman graph and then review 
relevant properties.

Suppose $\Sigma$ is a split system on $X$ and let  $\mathcal P(X)$
 denote the power set of $X$.
Then the {\em Buneman graph}  $G(\Sigma)$  associated to
$\Sigma$ is the subgraph of the
$|\Sigma|$-dimensional hypercube 
whose vertex set $V(\Sigma)=V(G(\Sigma))$  is the
set of all maps $\phi:\Sigma \to \mathcal P(X)$ such that
$\phi(S)\in S$ holds for all $S\in\Sigma$ and
$\phi(S)\cap \phi(S')\not=\emptyset$,
for all $S,S'\in \Sigma$. Note that $V(\Sigma)\not=\emptyset$
 since for all $x\in X$ the  {\em Kuratowski map} associated to
$x$ defined by putting 
$\phi_x:\Sigma \to \mathcal P(X)$: $S\mapsto S(x)$
is contained in $V(\Sigma)$.
The edge set $E(\Sigma)$
comprises all sets $\{\phi,\phi'\}\in {V(\Sigma)\choose 2}$
such that the difference set
$\Delta(\phi,\phi')=\{S\in \Sigma: \phi(S)\not=\phi'(S)\}$
contains precisely one element. 
We say that a split $S=A|B$ of $X$ is  
{\em Bu-displayed} by $G(\Sigma)$
if there exists a ``ladder'' $E'$ of parallel edges whose deletion disconnects
$G(\Sigma)$ into two connected components 
one of whose vertex sets contains $A$ and the
other $B$ (see e.\,g.\,\cite[Lemma 4.5]{DHKMS12} for details
where Bu-displayed is called displayed). 
Note that every split that is Bu-displayed by $G(\Sigma)$
is also a minimal cut of $G(\Sigma)$ and thus displayed by $G(\Sigma)$.
However the converse need not hold.

To illustrate these definitions let $X=\{1,\ldots,8\}$ 
and consider again the split system $\Sigma$ displayed by the
1-nested network depicted in Fig.~\ref{fig:example}(i).
Then the Buneman graph $G(\Sigma)$
associated to $\Sigma$ is depicted in Figure~\ref{fig:intro-network}(ii).

\subsection{Marguerites and Blocks}
In this section, we first focus on the Buneman graph of a 
maximal circular split system and then introduce and study the novel concept 
of a marguerite. We start with collecting some relevant
results. 

For $\Sigma$ a split system on $X$, the 
following five properties of $G(\Sigma)$ are well-known
(see e.\,g.\,\cite[Chapter 4]{DHKMS12}).
\begin{enumerate}
\item[(Bi)] The split system $\Sigma(G(\Sigma))$ 
Bu-displayed by $G(\Sigma)$ is $\Sigma$.
\item[(Bii)] For $\phi\in V(\Sigma)$  let $\mathrm{\min}(\phi(\Sigma))$ denote
the set-inclusion minimal elements in 
$\phi(\Sigma):=\{\phi(S)\,:\, S\in\Sigma\}$
and let $\Sigma^{(\phi)}$ denote 
the set of pre-images of $\mathrm{\min}(\phi(\Sigma))$ under $\phi$. 
Then a vertex $\psi\in V(\Sigma)$ is
adjacent with $\phi$ if and only if there exists some 
split $S^*\in \Sigma^{(\phi)}$
such that $\psi(S^*)=\overline{\phi(S^*)}$ and  $\psi(S)=\phi(S)$,
otherwise. In particular,
$|\Sigma^{(\phi)}|$ is the degree of $\phi$ in $G(\Sigma)$.
\item[(Biii)] In case $\Sigma$ contains all
trivial splits on $X$ then
$\Sigma$ is compatible if and only if, when
identifying each Kuratowski map $\phi_x$ with its underlying element
$x\in X$, $G(\Sigma)$ is a unrooted  phylogenetic tree on $X$ for which
$\Sigma(G(\Sigma))=\Sigma$ holds. Moreover, and up to
 isomorphism, $G(\Sigma)$ is unique.
\end{enumerate}

Note that for any two distinct compatible splits $S$ 
and $S'$ of $X$ there must exist 
a unique subset $A \in  S\cup S'$, say $A\in S$,  
such that $A \cap A'\not= \emptyset$ holds for all $A'\in S'$.
Denoting that subset by $ \max(S|S')$, we obtain

\begin{enumerate}
\item[(Biv)] For   $\Sigma_1,\Sigma_2\in \pi_0(\Sigma)$ distinct
we have $\max(S_1|S_2)= \max(S_1|S_2')$, for all $S_1\in\Sigma_1$ and all
$S_2,S_2'\in \Sigma_2$. 
In consequence, 
$\max(S_1|\Sigma_2):=\max(S_1|S_2)$ is well-defined 
where $S_1\in\Sigma_1$ and $S_2\in \Sigma_2$ \cite[Section 5]{DHKMS12}.
\end{enumerate}

A graph is called {\em 2-connected} if, after deletion of any of 
its vertices, it remains connected or is an isolated vertex \cite{DHKM11}. 
Calling a maximal
$2$-connected component of a graph $G$ a {\em block}
of $G$ and denoting the set of blocks of $G$ by $\frak{Bl}(G)$ we obtain

\begin{enumerate}
\item[(Bv)] The blocks 
of $G(\Sigma)$ are in 1-1 correspondence
with the connected components of $Incomp(\Sigma)$. 
More precisely the map
$
\Theta: \pi_0(\Sigma)\to \frak{Bl}(G(\Sigma)):\,\,\Sigma_0\mapsto
B(\Sigma_0):=\{\phi\in V(\Sigma): \phi(S)=\max(S|\Sigma_0)
\mbox{ holds for all } S\in\Sigma-\Sigma_0\}.
$
is a bijection \cite[Theorem 5.1]{DHKM11}.
\end{enumerate}

To illustrate these definitions, consider again the Buneman graph depicted in 
Figure~\ref{fig:intro-network}(ii) and the
splits $S=78|1\ldots 6$ and $S'=18|2\ldots 7$ both of which are 
 Bu-displayed by that graph.
Then for the marked vertex $\phi$, we have
$\phi(S)=\{7,8\}$. The block marked $\mathcal B_1$ in that figure  
corresponds via $\Theta$ to the connected component 
$\Sigma_0=\{S,S'\}$  and $\max(S'|\Sigma_0)=X-\{2,3,4\}$.

For the following assume that $k\geq 4$ and that $Y=\{X_1,\ldots, X_k\}$
is a partition of $X$. For
clarity of exposition, also assume that $|X_i|=1$ for all $1\leq i\leq k$
and that the unique element in $X_i$ is denoted by $i$.
Further, assume 
that $\sigma$ is the lexicographical ordering of $X$ where we put
$k+1:=1$. Let $\Sigma_k$ denote the maximal circular split system
displayed by $\sigma$ bar the trivial splits of $X$.
Since $\Sigma_k$ contains all 2-splits displayed by $\sigma$ it follows
that $|\pi_0(\Sigma_k)|=1$ and, so, $G(\Sigma_k)$ is a block
in view of Property (Bv). To better understand the structure of 
$B(\Sigma_k)$ consider for all  $1 \leq i \leq k$ and for all 
$0 \leq j < k-3$ the map:

\[
\begin{array}{r c l}
\phi_i^j : \Sigma_k \to \mathcal{P}(X):\,\,\,\, 
S \mapsto
\left\{
\begin{array}{r l}
\overline{S(i)} &\text{if } S(i) \subseteq [i-j, i] \\
S(i) &\text{otherwise.}
\end{array}
\right.
\end{array}
\]
For example for the $k=6,8$ the map $\phi_1^2$ is indicated
by a vertex in Figure~\ref{fig:marg}(i) and (ii), respectively.
\begin{figure}[h]
\begin{center}
\includegraphics[scale=0.6]{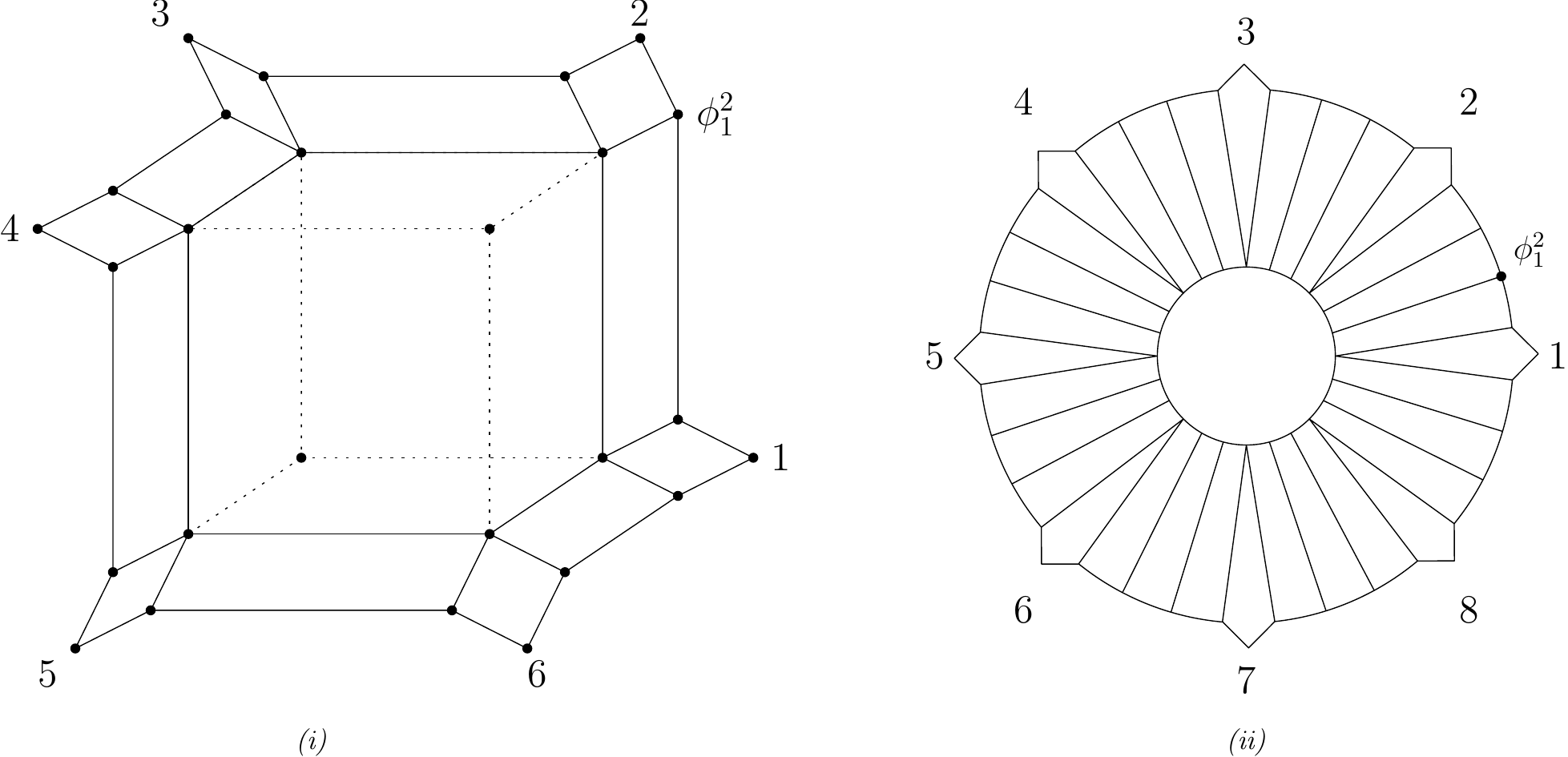}
\caption{\label{fig:marg} For $k=6$, we depict in (i) the Buneman 
graph $G(\Sigma_6)$ in terms of bold and dashed edges and the associated
$6$-marguerite $M(\Sigma_6)$ in terms of bold lines. In addition we indicated
the vertex $\phi_1^2$ of $G(\Sigma_6)$. We picture the $8$-marguerite in (ii)
and indicate again the vertex $\phi_1^2$. }
\end{center}
\end{figure}

To establish the next result, we associated to every element $i\in X$ 
the split systems
 $\Sigma(i)^+:=\{S\in \Sigma_k\,:\, S(i+1)=S(i)\not=S(i-1)\}$.
Then the 
partial ordering ``$\preceq_i$'' defined, for all $S,S'\in \Sigma_k$, by putting
$S\preceq_i S'$ if $|S(i)|\leq |S'(i)|$,
is clearly a total ordering of  $\Sigma(i)^+$ with minimal
element $S^+_i=[i,i+1]|X-[i,i+1]$ 

\begin{lemma}\label{lem:k-marguerite}
For any $k\geq 4$ the following statements hold:

\noindent (i) For  all $i\in\{1\ldots,k\}$ and all $0 \leq j < k-3$
the map $\phi_i^j$ is a vertex of $G(\Sigma_k)$, $\phi_i^{k-3}=\phi_{i+1}^0$
holds,
and $\Delta(\phi_i^j,\phi_i^{j+1})=\{[i-j-1,i]|X-[i-j-1,i]\}$.
In particular, $\{\phi_i^j,\phi_i^{j+1}\}$ is an edge in  $G(\Sigma_k)$.

\noindent (ii) For all $i\in\{1\ldots,k\}$ and all $1\leq j<k-3$, the map
\[
\begin{array}{r c l}
\psi_i^j : \Sigma_k \to \mathcal{P}(X):\,\,\,\, 
S \mapsto
\left\{
\begin{array}{r l}
\overline{\phi_i^j(S)} &\text{if } S=S^+_i \\
\phi_i^j(S) &\text{otherwise.}
\end{array}
\right.
\end{array}
\]
is a vertex in $G(\Sigma_k)$ that is adjacent with $\phi_i^j$. Moreover
$\psi_i^{k-3}=\psi_{i+1}^0$ and $\{\psi_i^j,\psi_i^{j+1}\}$ is an edge
in $G(\Sigma_k)$. 
\end{lemma}

\begin{proof}
(i) Suppose $i\in\{1\ldots,k\}$ and $0 \leq j < k-3$.
To see that $\phi_i^j\in V(\Sigma_k)$, we distinguish
between the cases that (a) $j=0$, (b) $j=k-3$, and (c) $1\leq j\leq k-4$. 
Let $i\in \{1,\ldots, k\}$.

Assume first that (a) holds and let $S\in \Sigma_k$.
Then  $\phi_i^0(S)=S(i)$ must hold since $\Sigma_k$ does not contain
trivial splits. Moreover, $\phi_i^0(S)=\overline{S(i)}$ holds if and
 only if $S(i) \subseteq \{i\}$ if and only if $S$ is the trivial split $i|X-i$.
Thus, $\phi_i^0$ is a vertex in $G(\Sigma_k)$ in this case. 

Assume next that (b) holds. We 
claim that $\phi_i^{k-3}=\phi_{i+1}^0$. Assume again that 
$S\in \Sigma_k$. Observe that since
$i-(k-3)\equiv i+3$ $(\negthickspace\negthickspace\mod k)$ we have
$S(i) \subseteq \{i-(k-3), \ldots, i\}$ if and only if 
$\{i+1,i+2\} \subseteq \overline{S(i)}$. 
We distinguish the cases that ($\alpha$) $S(i)=S(i+1)$
and ($\beta$) $S(i)\not=S(i+1)$.

Assume first that Case~($\alpha$) holds, that is,  $S(i)=S(i+1)$. 
Then $\{i+1,i+2\} \not\subseteq \overline{S(i)}$.
Combined with the observation made at the beginning of the
proof of this case, we obtain
$S(i)\not \subseteq \{i-(k-3), \ldots, i\}$ and, so,
$\phi_i^{k-3}(S)=S(i)=S(i+1)=\phi_{i+1}^0(S)$.

Next, assume that  Case~($\beta$) holds, that is, $S(i)\not=S(i+1)$.
Then $i+1\in \overline{S(i)}$. Since $S$ cannot be 
a trivial split it follows that $i+2\in\overline{S(i)}$
must hold too. Combined again with the observation made at the beginning
of the proof of this case,
it follows that  $S(i) \subseteq \{i-(k-3), \ldots, i\}$.
Thus, $\phi_i^{k-3}(S)=\overline{S(i)}=S(i+1)=\phi_{i+1}^0(S)$
which completes the proof of the claim.
In combination with Case~($\alpha$), $\phi_i^{k-3}\in G(\Sigma_k)$ follows. 

So assume that (c) holds.    
Combining (a) with (Bii) and the fact that 
$\phi_i^0(S)=\phi_i^1(S)$ for all $S\in \Sigma_k-\{S^+_i \}$ and 
$\phi_i^0(S^+_i)=\overline{\phi_i^1(S^+_i )}$, it follows that $\phi_i^1$
is a vertex of $G(\Sigma_k)$. Similar arguments imply that
if $\phi_i^l$ is a vertex in $G(\Sigma_k)$ then so is $\phi_i^{l+1}$.
This concludes the proof of Case (c).

That $\Delta(\phi_i^j,\phi_i^{j+1})=\{[i-j-1,i]|X-[i-j-1,i]\}$
holds for all $i\in\{1\ldots,k\}$ and $0 \leq j < k-3$ is an immediate
consequence of the construction.

(ii) Suppose $i\in\{1\ldots,k\}$ and $1\leq j<k-3$. Then
 $\psi_i^j$ must be a vertex of $G(\Sigma_k)$ that is adjacent with 
$\phi_i^j$ in view of (Bii) as $S^+_i \in \Sigma^{\phi_i^j}$. That 
$\psi_i^1=\psi_{i-1}^{k-3}$ is implied by the fact that the two
splits in which $\psi_i^{k-3}$ and $\psi_{i+1}^1$ differ from
$\phi_{i+1}^0$ are incompatible. That 
$\{\psi_i^j,\psi_i^{j+1}\}$ is an edge in $G(\Sigma_k)$ follows from the
fact that $\{\phi_i^j,\phi_i^{j+1}\}$ is an edge in $G(\Sigma_k)$.
\end{proof}

Bearing in mind Lemma~\ref{lem:k-marguerite}, we next associate 
to $G(\Sigma_k)$ the {\em $k$-marguerite} $M(\Sigma_k)$ on $X$,
that is, the subgraph of $G(\Sigma_k)$  induced 
by the set of maps $\phi_i^j$ and $\psi_i^l$
where $1\leq i\leq k$, $0\leq j<k-3$ and $1\leq l<k-3$. We 
illustrate this definition 
for $k=6,8$ in  Fig.~\ref{fig:marg}. Note that
if $k$ or $X$ are of no relevance to the discussion
then we will also refer to a $k$-marguerite on $X$ simply 
as a {\em marguerite}.

Clearly, $G(\Sigma_k)$ and $M(\Sigma_k)$ coincide for $k=4,5$. To be able to
shed light into the structure of $k$-marguerites for $k\geq 6$, we require
 some more terminology. Suppose 
$k\geq 4$ and  $i\in\{0,\ldots, k\}$. Then we call a 
vertex of $M(\Sigma_k)$ of 
the from $\phi_i^0$ an \emph{external vertex}. Moreover,
we call for all $0\leq j<k-3$ an edge of  $M(\Sigma_k)$  of the form 
$\{\phi_i^j,\phi_i^{j+1}\}$ an \emph{external edge}. 
Note that since $M(\Sigma_k)$ is in particular a subgraph
of the $|\Sigma_k|$-dimensional hypercube, any split in
 $\Sigma_k$ not of the form $i,i+1| X-\{i,i+1\}$ is Bu-displayed 
in terms of four parallel edges of $M(\Sigma_k)$ exactly two of 
which are external. 

\subsection{Gates}
In this section we establish that any partially-resolved 1-nested network
can be embedded into the Buneman graph associated to $\Sigma(N)$
thus allowing the bringing to bear of a wealth of results for the
Buneman graph to such networks. Of particular interest  to us
are gated subsets of $V(\Sigma)$ where a
subset $Y\subseteq Z$ of a (proper) metric space $(Z,D)$ is called 
a {\em gated} subset of $Z$ if there exists for every
$z\in Z$ a (necessarily unique) element $y_z\in Y$ such that
$D(y,z)=D(y,y_z)+D(y_z,z)$ holds for all $y\in Y$.  We refer 
to $y_z$ as the {\em gate} for $z$ in $Y$.
 
We start with associating a metric space to the Buneman graph of a split system.
Suppose  $\Sigma$ is a split system
on $X$ such that for all $x$ and $y$ in $X$ distinct there
exists some $S\in \Sigma$ such that $S(x)\not=S(y)$. Then
 the map
$D:V(\Sigma)\times V(\Sigma)\to \mathbb R_{\geq 0}:
(\phi,\phi')\mapsto  |\Delta(\phi,\phi')|$, is 
a (proper) metric on $V(\Sigma)$ (see 
e.\,g.\,\cite[page 52]{DHKMS12}) that is, $D$ attains $0$
only on the main diagonal, is symmetric, and satisfies
the triangle inequality. 

For $\Sigma$ a split system on $X$  and $\Sigma'\in \pi_0(\Sigma)$, 
the following two additional properties of the 
Buneman graph will be useful.

\begin{enumerate}
\item[(Bvi)] The map
\[
\begin{array}{r c l}
V(\Sigma') \to V(\Sigma): \,\,\,\,
\phi \mapsto 
(\begin{array}{r c l}
\tilde{\phi} : \Sigma \to \mathcal{P}(X):   \,\,\,\,
S \mapsto
\left\{
\begin{array}{r l}
\phi(S) &\text{ if } S \in \Sigma', \\
\mathrm{max}(S|\Sigma') &\text{ otherwise}
\end{array}
\right.)
\end{array}
\end{array}
\]
is an isometry between $G(\Sigma')$ and the block $B(\Sigma')$ of 
$G(\Sigma)$.

\item[(Bvii)] For every map $\phi\in V(\Sigma)$, the map $\phi_{\Sigma'}$
given by 
$$
\phi_{\Sigma'} : \Sigma(N) \to \mathcal{P}(X):   \,\,\,\,
S \mapsto
\left\{
\begin{array}{r l}
\phi(S) &\text{ if } S \in \Sigma' \\
\mathrm{max}(S|\Sigma') &\text{ otherwise,}
\end{array}
\right.
$$
\end{enumerate}
is the gate for $\phi$ in $B(\Sigma')$.
We denote by $Gates(G(\Sigma))$ the set of all
vertices $\phi$ of $G(\Sigma)$ for which there exists a block 
$B\in \frak{Bl}(G(\Sigma))$
such that $\phi$ is the gate for some $x\in X$ in $B$.



\begin{lemma}
Suppose $N$ is a 1-nested network on $X$. Then a block of $G(\Sigma(N))$ is
either a  cut-edge or contains precisely one marguerite. 
Moreover the gates of a marguerite 
$M$ in $G(\Sigma(N))$ are the maps $\tilde{\phi}$ where $\phi$ is an
external vertex of $M$. 
\end{lemma}

\begin{proof}
Suppose $\Sigma'\in \pi_0(\Sigma(N))$. Note that  $|\Sigma'|=1$ if and only if
$B(\Sigma')$ is a cut-edge of $G(\Sigma(N))$. So assume that 
$|\Sigma'|\geq 2$. Then $B(\Sigma')$ is a block of $G(\Sigma(N))$ and, so,
there exists a unique cycle $C$ of $N$ of
length $k\geq 4$ such that $\Sigma(C)=\Sigma'$. Let
$Y$ denote the partition of $X$ induced
by deleting all edges of $C$ and let $\Sigma'_Y$ denote the split system
on $Y$ induced by $\Sigma(C)$. Then $\Sigma'_Y$ is of the
form $\Sigma_k$ and, so, $G(\Sigma_Y')$ contains the
$k$-marguerite $M(\Sigma'_Y)$. Combined with Property (Bvi)
it follows that  $G(\Sigma(N))$ contains the marguerite $M(\Sigma'_Y)$
(or, more precisely, the graph obtained by replacing for every external vertex
$\phi_i^0$, $1\leq i\leq k$, the label $Y_i\in Y$ by the elements in $Y_i$). 

To 
see the remainder of the lemma suppose that $M$ is a marguerite
and assume that $k\geq 4$ such that $M=M(\Sigma_k)$. 
Let $Y=\{X_1,\ldots,X_k\}$ denote 
the partition of $X$ induced by $\Sigma_k$ and assume that  
$x\in X$. Then there must exist some $i\in\{1,\ldots,k\}$ such that
$x\in X_i$. Since $\phi_i^0$ is clearly the map 
\[
\begin{array}{r c l}
\phi_i^0 : \Sigma_k \to\mathcal{P}(X):\,\,\,\,
S=A|B \mapsto
\left\{
\begin{array}{r l}
A &\text{ if } X_i \subseteq A \\
B &\text{ if } X_i \subseteq B,
\end{array}
\right.
\end{array}
\]
Properties~(Bvi) and (Bvii) imply that
 $\tilde{ \phi_i^0} $ is the gate for $x$ in $M$.
\end{proof}

To be able to establish that any $1$-nested partially resolved  
network $N$ can be
embedded as a (not necessarily induced) subgraph
into the Buneman graph $G(\Sigma(N))$ associated to $\Sigma(N)$,
we require again more terminology. Suppose $N$ is a partially-resolved 
1-nested network and $v$ is a non-leaf vertex of $N$. Then 
$v$ is either incident with three or more cut-edges of $N$,  or 
there exists a cycle $C_v$  of $N$ 
that contains $v$ in its vertex set.
In the former case, we choose one of them and denote it by $e_v$.
In addition, we denote by $x_v\in X$ an element
such that $e_v$ is not contained in any path in $N$ 
from $x_v$ to $v$. In the latter case, we define $x_v$
to be an element in $X$ such that 
no edge of $C_v$ is contained in any path in $N$ from $v$ to $x_v$. 

\begin{theorem}\label{theo:const}
Suppose $N$ is a 1-nested partially-resolved network on $X$. Then the 
map $\psi:V(N)-X\to Gates(G(\Sigma(N)))$ defined by 
mapping every non-leaf vertex $v\in V(N)$ to the map
$$\xi(v):\Sigma(N)\to \mathcal P(X):\,\,\,\,
 S\mapsto
\left\{
\begin{array}{r l}
\max(S|\Sigma^*) &\text{ if } S\in \Sigma(N)-\Sigma^* \\
S(x_v) &\text{ else }
\end{array}
\right.
$$
is a bijection between the set of  non-leaf 
vertices of $N$ and the gates of $G(\Sigma(N))$
where $\Sigma^*=\{S_{e_v}\}$ if $v$ is contained in 
three or more cut-edges of $N$ and $\Sigma^*=\Sigma(C_v)^-$ else.
In particular, $\xi$ induces an embedding of $N$ into 
$G(\Sigma(N))$ by mapping each leaf $x$ of $N$ to the 
leaf $\phi_x$ of $G(\Sigma(N))$ and 
replacing for any two adjacent vertices $v$ and
$w$ of a cycle $C$ of $N$ of length $k$ the edge $\{v,w\}$ by the path 
$\phi_i^0:=\xi(v) ,\phi_i^1,\ldots, \phi_i^{k-3}:=\xi(w)$. 
\end{theorem}

\begin{proof}
Suppose $N$ is a 1-nested network and put $\Sigma=\Sigma(N)$. To see
that $\xi$ is well-defined suppose $v\in V(N)-X$. 
Then $v$ is either contained
in three or more cut-edges of $N$ or 
 $v$ is a vertex of some cycle $C$ of $N$. In the former case
we obtain $\{S_{e_v}\}\in \pi_0(\Sigma(N))$  and in the later we have
$C=C_v$ and $\Sigma(C_v)^-\in \pi_0(\Sigma)$. In either case, 
the definition of the element $x_v$ combined with Property (Bvii)
implies $\xi(v)\in Gates(G(\Sigma))$.

To see that $\xi$ is injective suppose $v$ and $w$ are two 
non-leaf vertices of $N$ such that $\xi(v)=\xi(w)$. 
Assume for contradiction that $v\not=w$. 
 It suffices to distinguish
between the cases that (i) $v$ and $w$ are contained in the same cycle, 
and that (ii) there exists a cut edge $e'$ on any path from $v$ to $w$.

To see that (i) cannot hold, suppose that $v$ and $w$ are vertices on a 
cycle $C$ of
$N$. Then, $S(x_v) = \max(S|\Sigma(C)^-)=S(x_w)$ must hold
 for the m-split $S$ 
obtained by deletion of the two edges of $C$ adjacent to $v$
which is impossible. Thus 
(ii) must hold. Hence, there must exist a cut-edge 
$e'$ on the path from $v$ to $w$.
Then  $\xi(v)(S_{e'}) \neq \xi(w)(S_{e'})$ follows which is again 
impossible. Thus, $\xi$  must be injective.



To see that $\xi$ is surjective suppose $g\in Gates(G(\Sigma))$.
Then there exists some $x_g\in X$ and some block $B\in \frak{Bl}(G(\Sigma))$
such that $g$ is the gate for $x_g$ in $B$. Let $\Sigma_B\in \pi_0(\Sigma(N))$ 
denote the connected component that, in view of Property (Bv) is in one-to-one
correspondence with $B$. If there exists a cycle $C$ of $N$ such
that $\Sigma(C)^-=\Sigma_B$ then let $v_g$ be a vertex of $N$ such that no edge
on any path from $v_g$ to $x_g$ crosses an edge of $C$. Then, by construction,
$\xi(v_g)=g$. Similar arguments show that $\xi(v_g)=g$ must hold if
$ \Sigma_B$ contains precisely one split and thus corresponds to a cut-edge
of $N$.
 Hence, $\xi$ is also surjective and thus bijective.

The remainder of the theorem is straight-forward.
\end{proof}

Theorem \ref{theo:const} implies that by carrying out steps 
(Ci) and (Cii) stated in 
Corollary~\ref{cor:construction}
any 1-nested partially-resolved 
network $N$ induces a 1-nested network $N(\Sigma(N))$  
such that  split system $\Sigma(N(\Sigma(N)))$ induced by $N(\Sigma(N))$
is the split system $\Sigma(N)$ induced by $N$.

\begin{corollary}\label{cor:construction}
Let $\Sigma$ be a split system on $X$ 
for which there exists a 1-nested network $N$ 
such that $\Sigma=\Sigma(N)$. Then we can obtain 
$N(\Sigma)$ from $G(\Sigma)$ by carrying out 
 steps (Ci) and (Cii):\\
\emph{(Ci)} For all $x\in X$ replace each leaf $\phi_x$ by $x$,
and \\
\emph{(Cii)} For all blocks $B$ of $G(\Sigma)$ that 
contain a $k$-marguerite $M$ for some 
$k\geq 4$, first add the edges 
$\{\phi_i^0,\phi_{i+1}^0\}$ for all $i \in \{1, \ldots, k\}$ 
where $k+1:=1$ and 
then delete all edges and vertices of $B$ not of the form $\phi_i^0$
for some $1\leq i \leq k$.
\end{corollary}


We next show that  even if the circular
split system under consideration does not satisfy the assumptions of 
Corollary~\ref{cor:construction}, steps (Ci) and (Cii) still give rise to
a, in a well-defined sense, optimal 1-nested network

\begin{theorem}\label{theo:buneman}
Let $\Sigma$ be a circular split system on $X$ that contains all trivial splits 
on $X$. Then $N(\Sigma)$
is a 1-nested network such that:\\
\emph{i)} $\Sigma \subseteq \Sigma(N)$,\\
\emph{ii)} $|\Sigma(N)|$ is minimal among the 1-nested network
 satisfying \emph{i)},\\
\emph{iii)} A vertex $v$ of a cycle $C$ of $N$ is partially resolved if
and only if the split displayed by the edges of $C$ incident with 
$v$ belongs to $\Sigma$.

\noindent Moreover $N$ is unique up to isomorphism and partial-resolution.
\end{theorem}

\begin{proof}
(i) $\&$ (ii): Suppose for contradiction that there exists a 
1-nested network $N'$
such that $\Sigma\subseteq \Sigma(N')$
and $|\Sigma(N')|< |\Sigma(N(\Sigma))|$.
Without loss of generality, we may assume that $N'$ is such that 
$|\Sigma(N')|$ is as small as possible. Moreover, we may assume 
without loss of generality that $N'$ and $N(\Sigma)$ are both 
maximal partially-resolved. To obtain the required contradiction,
we employ Corollary~\ref{cor:equivalence} 
to establish that $N'$ and $N(\Sigma)$ 
are isomorphic. 

Since $\Sigma\subseteq \I(\Sigma)$ it is clear that $\I(\Sigma)$
contains all trivial splits of $X$.
Furthermore, since $\Sigma$ is circular, Corollary~\ref{cor:cii}(i) 
implies that $\I(\Sigma)$ is circular. Since $\I(\Sigma)$ is clearly 
$\I$-intersection closed and, by Property (Bi), $\I(\Sigma)$ is the split 
system Bu-displayed by $G(\I(\Sigma))$ it follows that $\I(\Sigma)$ 
comprises all splits  displayed by $N(\I(\Sigma))$. Hence, by 
Corollary~\ref{cor:equivalence}, up to isomorphism and partial-resolution,
$N(\I(\Sigma))$ is the unique 1-nested network for which the
displayed split system is  $\I(\Sigma)$.

We claim that $\I(\Sigma)=\Sigma(N')$ holds too.
By Corollary~\ref{cor:cii}(iii), we have   
$\I(\Sigma)\subseteq \Sigma(N')$. 
To see the converse set inclusion
assume that $S\in \I(\Sigma)$. Then $S$ is either induced by 
(a) a cut-edge of $N'$ or (b) $S$ is not an m-split and 
there exists a cycle $C$ of $N'$ that displays $S$. 
In case of (a) holding, $S\in \Sigma$ follows 
by the minimality of $|\Sigma(N')|$. So assume that (b) holds.
Then there must exist some connected component 
$\Sigma_C\in \pi_0(\Sigma)$ that displays $S$. Hence, by Property (Bv),
there exists some block $B_C\in \frak{Bl}(\Sigma)$
such that the split system Bu-displayed by  $B_C$ is $\Sigma_C$.
Hence, $\Sigma_C$ is also displayed by $N(\Sigma)$. Since, as observed above
$\Sigma(N(\Sigma))=\I(\Sigma)$ we also have $\Sigma(N')\subseteq \I(\Sigma)$
the claim follows.

(iii) Suppose  $C$ is a cycle of $N$ and $v$ is a vertex of $C$.
Assume first that $v$ is partially resolved. Then there exists a cut-edge
$e$ of $N$ that is incident with $v$. Note that the split $S_e$ displayed by
$e$ is also displayed by the two edge of $C$ incident with $v$.
In view of Property (Bi) and, implied by (Ci) and (Cii), that
 the cut-edges of $N$ are in 1-1 correspondence with the
cut-edges of $G(\Sigma)$ we obtain $ S_e\in\Sigma$. 

To see the converse assume that $e_1$ and $e_2$ are the two edges of $C$
incident with $v$ such that the split $S$ displayed by $\{e_1,e_2\}$ is
contained in $\Sigma$. Then $S$ is compatible with all splits in
$\Sigma-\{S\}$. By Property (Bv), it follows that there exists a cut-edge 
$e$ in $G(\Sigma)$ such that $S_e=S$. Combined with  (Ci) and (Cii)
it follows that $v$ is partially resolved.
\end{proof}

\section{Open Questions}\label{sec:open-probs}

In this paper, we have started to investigate the interplay
between Buneman graphs and 1-nested networks. Although our results
are encouraging involving non-trivial characterizations, 
numerous questions that might be of interest have remained unanswered. 
For example, regarding Corollary~\ref{cor:equivalence} 
what is the minimal size of $\Sigma$ that allows one to, in our sense, uniquely 
recover $\Sigma(N)$?  Also, is it possible to characterize 
split system induced by level-2 networks (i.\,e.\,networks 
obtained from level-1 networks by adding
a cord to a cycle)? Finally, a number of 
reconstruction algorithms to reconstruct rooted level-1 networks
try and infer them from a collection rooted binary phylogenetic 
trees on three leaves. Such trees are generally referred to as triplets
and in real biological studies it is generally to much to hope for that
a set of triplets contains {\em all} triplets induced by 
the (unknown) underlying network. One way to overcome this problem is 
to employ triplet inference rules. Such rules 
 are well-known for rooted phylogenetic trees
but are missing for general level-1 networks. The question therefore
becomes if the work presented here combined with results on closure
obtained in \cite{GHW08}  might provide a starting point 
for developing such rules.\\

\noindent {\bf Acknowledgments}
KTH and PG thanks the London Mathematical Society for supporting a visit
of KTH to PG where some of the ideas for the paper were conceived.

\bibliographystyle{abbrv}

\bibliography{bibliography}{}

\end{document}